\documentclass{article}
\usepackage[journal=APDE,lang=british]{ems-journal-jems} %% change to `american' if you use American English

\usepackage{textcomp}
\usepackage{mathrsfs}

\usepackage{chngcntr}
\usepackage{apptools}
\usepackage{verbatim}
\usepackage{tikz-cd}
\usepackage{subcaption}

\newtheorem{theorem}{Theorem}

\newtheorem{proposition}{Proposition}
\newtheorem{lemma}{Lemma}

\theoremstyle{definition}

\theoremstyle{remark}
\newtheorem{remark}{Remark}
\newtheorem{corollary}{Corollary}

\numberwithin{equation}{section}

\newcommand\restr[2]{{% we make the whole thing an ordinary symbol
		\left.\kern-\nulldelimiterspace % automatically resize the bar with \right
		#1 % the function
		\vphantom{\big|} % pretend it's a little taller at normal size
		\right|_{#2} % this is the delimiter
}}

\numberwithin{equation}{section}

\numberwithin{theorem}{section}
\numberwithin{proposition}{section}
\numberwithin{lemma}{section}
\numberwithin{remark}{section}
\numberwithin{exercise}{section}
\numberwithin{corollary}{section}

\newcounter{desccount}

\begin{document}

%------
% Insert the title of your paper and (if necessary)
% a short title for the running head.
%------
\title{Inertial manifolds via spatial averaging: a control-theoretic perspective}
\titlemark{Inertial manifolds and nonstationary quadratic optimization}

%------

%%%% Pls fill in all fields for each author
%%%% Label the authors by their position in the authors' list using {}
%%%% Pls look for the MR Author ID in the metadata.xml file in the source bundle
%%%% ORCID is only to be added if provided by the author
%%%% Abbreviate first names for the running head
\emsauthor{1}{
	\givenname{Mikhail}
	\surname{Anikushin}
	%\mrid{1234567}
	\orcid{0000-0002-7781-8551}}{M.M.~Anikushin}
%%%% Repeat the same fields for each numbered author
%\emsauthor*{3}{
%	\givenname{Someone}
%	\surname{Else}
%	\mrid{}
%	\orcid{}}{S.~Else}

%%%% Please provide detalied address info for each author
%%%% Use the same numbering as for \emsauthor above
%%%% Please look up the ROR ID of your institute here: https://ror.org
\Emsaffil{1}{
	\pretext{}
	\department{Department of Applied Cybernetics}
	\organisation{Faculty of Mathematics and Mechanics, St Petersburg University}
	\rorid{01a2bcd34}
	\address{Universitetskiy prospekt 28}
	\zip{198504}
	\city{Peterhof}
	\country{Russia}
	\posttext{}
	\affemail{demolishka@gmail.com}
	}
%%%% Repeat the same fields for each numbered author
%%%% If some author has multiple affiliations, repeat the fields for each affiliation
%%%% Number the affiliations using {}
	
%% If a corresponding author is needed, use \emsauthor* instead of \emsauthor

%%%% Appendix authors are treated analogously with the only difference 
%%%% that they do not have the abbreviated name field
	
%%%% If a regular author is an Appendix author as well, repeat their data 
%%%% using  \Appendixauthor* instead of \Appendixauthor

%\Appendixauthor*{2}{
%	\givenname{Mikhail}
%%	\surname{Anikushin}
%	\mrid{1234567}
%	\orcid{0000-0001-0002-0003}}
%\Appendixaffil{2}{
%	\pretext{}
%	\department{}
%	\organisation{}
%	\rorid{}
%	\address{}
%	\zip{}
%	\city{}
%	\country{}
%	\posttext{}
%	\affemail{}
%	\furtheremail{}}

%------
% Add MSC 2020 codes according to www.ams.org/msc/msc2020.html.
% Secondary codes (in square brackets) are optional.
%------
\classification[37L25]{35B42}
%------
% Add a list of keywords. Only capitalise those keywords that start with a proper name.
%------

%------
% Insert your abstract.
%------
\begin{abstract}
	We develop a functional-analytical machinery for studying the quadratic regulator problem arising from spectra perturbations of infinite-dimensional dynamical systems. In particular, we are interested in applications to inertial manifolds theory. For certain nonautonomous Hamiltonian systems associated with such problems, we show the existence and uniform nonoscillation of stable Lagrangian bundles. This is done within the context of the classical frequency condition for stationary problems, as well as for nonstationary problems arising under the conditions of the Spatial Averaging Principle of J.~Mallet-Paret and G.R.~Sell.
\end{abstract}

\maketitle
\tableofcontents
%------
% INSERT THE BODY OF THE PAPER HERE (except
% acknowledgments, funding info and bibliography)
%------

\section{Introduction}

In the study of dissipative dynamical systems one is interested in understanding dynamics in a \textit{neighborhood} of an attractor. However, most of well-studied problems do not take into account topological nuances contained in the neighborhood and deal only with internal dynamics on the attractor itself\footnote{Therefore even the attraction is ignored.}. In this regard, the Bendixson criterion for attractors (see R.A.~Smith \cite{Smith1986HD}; M.Y.~Li and J.S.~Muldowney \cite{LiMuldowney1995LowBounds}) can be considered as a rare gem in the attractor theory, which furthermore provides an effective criterion for global stability when being synthesized with dimension estimates (see M.Y.~Li and J.S.~Muldowney \cite{LiMuldowney1996SIAMGlobStab}; N.V.~Kuznetsov and V.~Reitmann \cite{KuzReit2020}; our works \cite{Anikushin2023Comp, Anikushin2023LyapExp} and \cite{AnikushinRomanov2023FreqConds, AnikushinRomanov2024EffEst} (joint with A.O.~Romanov)).

More complex problems related to dynamics in a neighborhood are concerned with the existence of inertial manifolds. These are finite-dimensional positively invariant manifolds containing the global attractor and exponentially attracting\footnote{This is the so-called \textit{exponential tracking} property. As shown by the present author in \cite{Anikushin2020Geom}, in all known situations, besides the inertial manifold itself, there exists a strongly stable foliation being a bundle over the inertial manifold and the tracking trajectory is obtained via projection of the bundle (i.e. along the stable leaves).} all trajectories of the system by trajectories lying on the manifold. In the presence of low-dimensional inertial manifolds one can visualize and effectively study dynamics both rigorously and numerically that is especially relevant in high or infinite dimensions (see R.A.~Smith \cite{Smith1992}; our papers \cite{AnikushinAADyn2021, Anikushin2020Geom} and \cite{AnikushinRom2023SS} (joint with A.O.~Romanov)). We refer to the surveys of S.~Zelik \cite{Zelik2014, ZelikAttractors2022} for more discussions in the context of attractors for PDEs and finite-dimensional reduction.

In our works \cite{Anikushin2020Geom, Anikushin2020FreqParab, Anikushin2020FreqDelay, Anikushin2022Semigroups}, it is shown that most of known constructions of inertial manifolds are related to the stationary case of an infinite-horizon quadratic regulator problem. Here \cite{Anikushin2020Geom} presents a geometric theory for cocycles in Banach spaces\footnote{However, known applications are directly related to the possibility of considering a given problem in a proper Hilbert space setting such that the Banach space is continuously embedded into it and dynamics is smoothing in finite time w.r.t. the embedding (see \cite{Anikushin2020Geom} for details). Appropriate developments of this approach in a ``purely Banach space setting'' should be concerned with appropriate developments of the (no longer quadratic) regulator problem in Banach spaces, which are unknown to us. In the case of stability problems, some partial progress in finite dimensions is done in the recent study of A.V.~Proskurnikov, A.~Davydov and F.~Bullo \cite{ProskurnikovDavydovBullo2023}.} based on (indefinite) quadratic Lyapunov-like functionals and variants of the Frequency Theorem \cite{Anikushin2020FreqParab, Anikushin2020FreqDelay} are used as analytical tools to provide frequency conditions (see \eqref{EQ: FreqConditionStationaryGeneral} and \eqref{EQ: StatCaseSmithFreqCond} below) guaranteeing the existence of such functionals for particular problems. This approach allowed us to unify many scattered works in the field, including parabolic and delay equations, and extend possible applications (see \cite{Anikushin2020FreqParab, Anikushin2020FreqDelay} for discussions).

However, the Spatial Averaging Principle suggested by J.~Mallet-Paret and G.R.~Sell \cite{MalletParetSell1988SA} to construct inertial manifolds for certain scalar reaction-diffusion equations in 2D and 3D domains, where the Spectral Gap Condition is violated, goes above stationary optimization problems, although its usual geometric realization via a single quadratic cone is covered by the geometric theory from \cite{Anikushin2020Geom}. Recent progress in developing the principle is discussed in the survey of A.~Kostianko et al. \cite{KostiankoZelikSA2020} and it is related to the works of the authors. In particular, appropriate developments of the method allowed them to show the existence of inertial manifolds for complex 3D Ginzburg-Landau equations \cite{KostiankoSunZelik2021IMGinzburgLandau} via the \textit{spatio-temporal} averaging introduced by A.~Kostianko in \cite{Kostianko2020}.

In \cite{Anikushin2020FreqParab}, we posed the problem of establishing relations between the Spatial Averaging Principle and nonautonomous Hamiltonian systems associated with appropriate quadratic optimization problems in regard with investigations done by R.~Fabbri, R.~Johnson and C.~N\'{u}\~{n}ez \cite{FabbriJohnsonNunez2003YakFT} in finite dimensions. In the present paper we establish some of such relations. Namely, we prove that under the conditions of the principle, the corresponding Hamiltonian systems admit uniformly nonoscillating stable Lagrangian bundles. However, we firstly start by studying the stationary optimization and establishing similar results which are seem to be new even in this case.

For convenience we treat only the context of the classical work \cite{MalletParetSell1988SA} which is reflected in that all the quadratic forms are bounded in the main space. Although appropriate developments of the general approach for systems related to semilinear parabolic problems is straightforward, more delicate cases arise from studying delay equations. In \cite{Anikushin2020FreqDelay}, the well-posedness and resolution of the optimization problem in the stationary case is related to structural properties of solutions to inhomogeneous Cauchy problems and avoids utilizing properties of Hamiltonian systems. Since the corresponding quadratic forms involve operators which are defined only on a Banach space, it is unclear to us, how to deduce (or better, in what sense to understand) the corresponding Hamiltonian systems.

Moreover, one of key features of the corresponding Hamiltonian systems in infinite dimensions is the \textit{rigidity of Lagrangian structures}. Namely, since solutions for general initial data do not exist, one cannot take an arbitrary Lagrangian subspace and track its evolution as one often does in finite dimensions (see R.~Johnson et al. \cite{JohnsonObayaNovoNunezFabbri2016}). Thus there exist only stable Lagrange bundles and the application of geometric methods is forbidden for establishing their existence. In this regard, Lyapunov-Perron operators seems like a natural analytical machinery for studying the existence problem. For applications, it is also necessary to establish uniform nonoscillation of such bundles and here we leave open applications of general geometric methods for such kind of problems.

Let us also mention that for problems related to uniform exponential stability, the nonoscillation immediately follows from our construction of stable Lagrange bundles; see Corollaries \ref{COR: StatDichLagrangeSubsFredholmPair} (with $j=0$) and \ref{COR: SpatAvrgLagrangeSubsFredholmPair} (with $N=0$).

This paper is organized as follows. In Section \ref{SEC: Preliminaries} we develop preliminary theory which is concerned with projectors and Lagrange subspaces (see Section \ref{SUBSEC: ProjectorsAndLagrangeSubspaces}) and Lyapunov-Perron operators for exponential dichotomies of dual stationary systems (see Section \ref{SUBSEC: LyapPerronStatDich}). In Section \ref{SEC: LagrangeBundlesStationaryCase}, we show the existence and nonoscillation of stable Lagrange subspaces for autonomous Hamiltonian systems related to stationary quadratic optimization under the classical frequency condition and $L_{2}$-controllability. In Section \ref{SEC: SpatAvgNonstatOpt}, we show the existence and uniform nonoscillation of stable Lagrange bundles for nonautonomous Hamiltonian systems arising within the conditions of the Spatial Averaging Principle of J.~Mallet-Paret and G.R.~Sell.
\section{Preliminaries}
\label{SEC: Preliminaries}
\subsection{Projectors and Lagrange subspaces}
\label{SUBSEC: ProjectorsAndLagrangeSubspaces}
In this section, we expose some basic facts about projectors in a Hilbert space $\mathbb{H}$ and then establish some useful relations between Lagrange subspaces and direct sum decompositions.

Recall that a \textit{projector} in $\mathbb{H}$ is a bounded linear mapping $\Pi \colon \mathbb{H} \to \mathbb{H}$ satisfying $\Pi^{2} = \Pi$. For a given projector $\Pi$, the space $\mathbb{H}$ decomposes into a direct sum\footnote{When talking about direct sum decompositions, we always assume that the corresponding subspaces are closed.} as $\mathbb{H} = \operatorname{Ran}\Pi \oplus \operatorname{Ker}\Pi$ of the range and kernel of $\Pi$ respectively. Conversely, given a direct sum decomposition $\mathbb{H} = \mathbb{H}_{\sharp} \oplus \mathbb{H}_{\flat}$ there exists a unique projector $\Pi_{\sharp}$ with $\operatorname{Ran}\Pi_{\sharp} = \mathbb{H}_{\sharp}$ and $\operatorname{Ker}\Pi_{\sharp} = \mathbb{H}_{\flat}$. One says that $\Pi_{\sharp}$ is a projector onto $\mathbb{H}_{\sharp}$ along $\mathbb{H}_{\flat}$. There is also the complementary projector $\Pi_{\flat} := \operatorname{Id}_{\mathbb{H}} - \Pi_{\sharp}$ which projects onto $\mathbb{H}_{\flat}$ along $\mathbb{H}_{\sharp}$. If the subspaces $\mathbb{H}_{\sharp}$ and $\mathbb{H}_{\flat}$ are orthogonal, $\Pi_{\sharp}$ and $\Pi_{\flat}$ are said to be \textit{orthogonal projectors}. It is not hard to see that orthogonal projectors are exactly the self-adjoint ones.

Let $|\cdot|_{\mathbb{H}}$ be the norm in $\mathbb{H}$. By the Bounded Inverse Theorem, there is an equivalent norm given by $\|v\|^{2}:= |\Pi_{\sharp} v|^{2}_{\mathbb{H}} + |\Pi_{\flat}v|^{2}_{\mathbb{H}}$ for $v \in \mathbb{H}$ associated with any direct sum decomposition $\mathbb{H} = \mathbb{H}_{\sharp} \oplus \mathbb{H}_{\flat}$.

From $\Pi^{2} = \Pi$ it is clear that the adjoint $\Pi^{*}$ of $\Pi$ is also a projector in $\mathbb{H}$. It is not hard to see that $\operatorname{Ker}\Pi^{*} = (\operatorname{Ran}\Pi)^{\bot}$, where $\bot$ denotes the orthogonal complement in $\mathbb{H}$.

Thus, for a given direct sum decomposition $\mathbb{H} = \mathbb{H}_{\sharp} \oplus \mathbb{H}_{\flat}$ there is an associated decomposition $\mathbb{H} = \mathbb{H}_{\flat}^{\bot} \oplus \mathbb{H}_{\sharp}^{\bot}$ which we call \textit{dual}. Here the projector onto $\mathbb{H}_{\flat}^{\bot}$ along $\mathbb{H}_{\sharp}^{\bot}$ is given by $\Pi^{*}_{\sharp}$.

The following lemma is valid in Banach spaces.
\begin{lemma}
	\label{LEM: DirectSumDecompositionCorrective}
	Let $\mathbb{H} = \mathbb{H}_{\sharp} \oplus \mathbb{H}_{\flat}$ and let $\Pi_{\sharp}$ be the projector onto $\mathbb{H}_{\sharp}$ along $\mathbb{H}_{\flat}$. Suppose $\mathbb{L}$ is a closed subspace. Then the following statement are equivalent:
	\begin{enumerate}
		\item[1.] $\Pi_{\sharp} \colon \mathbb{L} \to \mathbb{H}_{\sharp}$ is a bijection;
		\item[2.] $\mathbb{H} = \mathbb{L} \oplus \mathbb{H}_{\flat}$.
	\end{enumerate}
\end{lemma}
\begin{proof}
	To show that item 1 implies item 2, let $\Pi^{-1}$ be the inverse to $\Pi_{\sharp} \colon \mathbb{L} \to \mathbb{H}_{\sharp}$ which is bounded by the Bounded Inverse Theorem. It is clear that $\Pi_{\mathbb{L}} := \Pi^{-1} \Pi_{\sharp}$ is a bounded projector with $\operatorname{Ran}\Pi_{\mathbb{L}} = \mathbb{L}$ and $\operatorname{Ker}\Pi_{\mathbb{L}} = \mathbb{H}_{\flat}$. 
	
	For the converse implication, we have $\mathbb{L} \cap \mathbb{H}_{\flat} = \{0\}$ and $\Pi_{\sharp} \Pi_{\mathbb{L}} = \Pi_{\sharp}$, where $\Pi_{\mathbb{L}}$ is the projector onto $\mathbb{L}$ along $\mathbb{H}_{\flat}$. From the first property, we get that $\Pi_{\sharp} \colon \mathbb{L} \to \mathbb{H}_{\sharp}$ is injective and the second property guarantees that the map is surjective. The proof is finished.
\end{proof}

The following simple lemma will be helpful to establish that intersections of certain subspaces are finite-dimensional.
\begin{lemma}
	\label{LEM: IntersectionDimensionLemma}
	Suppose that\footnote{Here ``f'' stands for ``forbidden'' and ``a'' stands for ``admissible''.} $\mathbb{H} = \mathbb{H}_{f} \oplus \mathbb{H}_{a}$ with $\dim\mathbb{H}_{a} =: j < \infty$ and $\mathbb{L}$ is a subspace in $\mathbb{H}$ such that $\mathbb{L} \cap \mathbb{H}_{f} = \{0\}$. Then
	\begin{equation}
		\dim \mathbb{L} \leq j.
	\end{equation}
\end{lemma}
\begin{proof}
	From the hypotheses we get that the projector onto $\mathbb{H}_{a}$ along  $\mathbb{H}_{f}$ is injective on $\mathbb{L}$. From this the conclusion is obvious.
\end{proof}

For a Hilbert space $\mathbb{H}$, we will use the natural complex structure $J$ on $\mathbb{H} \times \mathbb{H}$ given by $J(v,\eta) := (-\eta,v)$ for any $(v,\eta) \in \mathbb{H} \times \mathbb{H}$. Recall that a subspace $\mathbb{L}$ in $\mathbb{H} \times \mathbb{H}$ is called \textit{isotropic} if $J \mathbb{L} \subset \mathbb{L}^{\bot}$ and \textit{Lagrange} if $J \mathbb{L} = \mathbb{L}^{\bot}$.

Let $\Lambda(\mathbb{H})$ be the \textit{Lagrangian-Grassmanian} in $\mathbb{H}\times \mathbb{H}$, i.e. the set of all Lagrange subspaces in $\mathbb{H} \times \mathbb{H}$. It is endowed with the metric given for any $\mathbb{L}_{1},\mathbb{L}_{2} \in \Lambda(\mathbb{H})$ by
\begin{equation}
	\label{EQ: MetricLagrangianGrassmanian}
	d(\mathbb{L}_{1},\mathbb{L}_{2}) := \| \Pi_{\mathbb{L}_{1}} - \Pi_{\mathbb{L}_{2}} \|,
\end{equation}
where $\Pi_{\mathbb{L}_{1}}$ and $\Pi_{\mathbb{L}_{2}}$ are the orthogonal projectors onto $\mathbb{L}_{1}$ and $\mathbb{L}_{2}$ respectively and $\|\cdot\|$ is the operator norm in $\mathcal{L}(\mathbb{H} \times \mathbb{H})$.

It is clear that for any $\mathbb{L} \in \Lambda(\mathbb{H})$ and any self-adjoint bounded operator $P \in \mathcal{L}(\mathbb{L})$ on $\mathbb{L}$ the subspace
\begin{equation}
	\label{EQ: LangrGrassCharSAOp}
	\mathbb{L}_{P} := \{ v + JPv \ | \ v \in \mathbb{L} \}
\end{equation}
is Lagrange. Let $\mathcal{O}_{\mathbb{L}}$ consist of all such subspaces. It turns out that $\mathcal{O}_{\mathbb{L}}$ is an open subset in $\Lambda(\mathbb{H})$ and it is isomorphic to the set of all subspaces $\mathbb{L}_{0}$ which are transversal to $\mathbb{L}^{\bot}$, i.e. $\mathbb{L}_{0} + \mathbb{L}^{\bot} = \mathbb{H} \times \mathbb{H}$ (see Proposition 2.21 in \cite{Furutani2004FLG}). It can be shown that $\Lambda(\mathbb{H})$ is a differentiable manifold modeled on the space of bounded self-adjoint operators in $\mathbb{H}$.

For given subspaces $\mathbb{L}_{1}$ and $\mathbb{L}_{2}$ in $\mathbb{H}$, we consider $\mathbb{L}_{1} \times \mathbb{L}_{2}$ as the subspace in $\mathbb{H} \times \mathbb{H}$ by putting
\begin{equation}
	\mathbb{L}_{1} \times \mathbb{L}_{2} := \{ (v,\eta) \in \mathbb{H} \times \mathbb{H} \ | \ v \in \mathbb{L}_{1}, \eta \in \mathbb{L}_{2} \}.
\end{equation}
Moreover, $\{ 0_{\mathbb{H}}\}$ denotes the zero subspace in $\mathbb{H}$.

Now let us establish a series of useful statements.
\begin{proposition}
	\label{PROP: LagrangeThroughAProjector}
	Suppose $\Pi$ is a projector in $\mathbb{H}$. Then the subspace 
	\begin{equation}
		\mathbb{L} := \{ (v,\eta) \in \mathbb{H} \times \mathbb{H} \ | \ v \in \operatorname{Ran}\Pi, \ \eta \in \operatorname{Ker}\Pi^{*} = (\operatorname{Ran}\Pi)^{\bot} \}
	\end{equation}
	is Lagrange.
\end{proposition}
\begin{proof}
	Suppose that $(v_{0},\eta_{0}) \in \mathbb{H} \times \mathbb{H}$ is orthogonal ot $\mathbb{L}$. Since $\mathbb{L}$ splits into the orthogonal sum of $\operatorname{Ran}\Pi \times \{0_{\mathbb{H}}\}$ and $\{0_{\mathbb{H}}\} \times (\operatorname{Ran}\Pi)^{\bot}$, this is equivalent to 
	\begin{equation}
		\begin{split}
			&\langle v_{0}, v\rangle_{\mathbb{H}} = 0 \textit{ for any } v \in \operatorname{Ran}\Pi \text{ and },\\
			&\langle \eta_{0}, \eta\rangle_{\mathbb{H}} = 0 \textit{ for any } \eta \in (\operatorname{Ran}\Pi)^{\bot}
		\end{split}
	\end{equation}
	or that $v_{0} \in (\operatorname{Ran}\Pi)^{\bot}$ and $\eta_{0} \in \operatorname{Ran}\Pi$. In other words, $(v_{0},\eta_{0}) \in J \mathbb{L}$. The proof is finished.
\end{proof}

For $\mathbb{H} = \mathbb{R}^{n}$, it is clear that $\mathbb{L}$ is Lagrange if and only if $\mathbb{L}$ is isotropic and $\dim \mathbb{L} = n$, i.e. $\mathbb{L}$ is an isotropic subspace of maximal dimension. The following theorem extends this idea to infinite dimensions. Here having ``maximal dimension'' for an isotropic subspace is treated as being the graph over a direct sum decomposition into Lagrange subspaces. This is a key result which will be used further to establish that some constructed subspaces are Lagrange.

\begin{theorem}
	\label{TH: LagrangeSubspaceCriterion}
	Suppose $\mathbb{H} \times \mathbb{H} = \mathbb{H}_{\sharp} \oplus \mathbb{H}_{\flat}$, where $\mathbb{H}_{\sharp}$ and $\mathbb{H}_{\flat}$ are Lagrange subspaces. Let $\mathbb{L} \subset \mathbb{H} \times \mathbb{H}$ be any closed isotropic subspace such that the projector $\Pi_{\sharp}$ onto $\mathbb{H}_{\sharp}$ along $\mathbb{H}_{\flat}$ is bijective on $\mathbb{L}$. Then $\mathbb{L}$ is Lagrange.
\end{theorem}
\begin{proof}
	Since $J$ is unitary, we have that $\mathbb{H} \times \mathbb{H} = J\mathbb{H}_{\sharp} \oplus J\mathbb{H}_{\flat}$ and $\Pi^{J}_{\sharp} := J\Pi_{\sharp}J^{-1}$ is the projector onto $J\mathbb{H}_{\sharp}$ along  $J\mathbb{H}_{\flat}$. 
	
	Let $\Pi_{\flat}$ be the projector onto $\mathbb{H}_{\flat}$ along $\mathbb{H}_{\sharp}$. Since $\mathbb{H}_{\sharp}$ and $\mathbb{H}_{\flat}$ are Lagrange, we have $J\mathbb{H}_{\sharp} = (\mathbb{H}_{\sharp})^{\bot}$ and $J\mathbb{H}_{\flat} = (\mathbb{H}_{\flat})^{\bot}$. In other words, the decomposition $\mathbb{H} \times \mathbb{H} =  J\mathbb{H}_{\flat} \oplus J\mathbb{H}_{\sharp}$ is dual to the decomposition $\mathbb{H} \times \mathbb{H} = \mathbb{H}_{\sharp} \oplus \mathbb{H}_{\flat}$. Consequently,
	\begin{equation}
		\Pi^{J}_{\sharp} = J\Pi_{\sharp}J^{-1} = \Pi^{*}_{\flat}.
	\end{equation}
    From this and since $\Pi_{\sharp} \colon \mathbb{L} \to \mathbb{H}_{\sharp}$ is a bijection, $\Pi^{*}_{\flat} \colon J\mathbb{L} \to J\mathbb{H}_{\sharp}$ is also a bijection.
    
    Let us show that $\Pi^{*}_{\flat} \colon \mathbb{L}^{\bot} \to J\mathbb{H}_{\sharp}$ is injective. Suppose that $z \in \mathbb{L}^{\bot}$ is such that $\Pi^{*}_{\flat} = 0$. Then $z \bot \operatorname{Ran}\Pi_{\flat} = \mathbb{H}_{\flat}$ and, consequently, $z \bot \mathbb{L} \oplus \mathbb{H}_{\flat}$. By Lemma \ref{LEM: DirectSumDecompositionCorrective}, $\mathbb{L} \oplus \mathbb{H}_{\flat} = \mathbb{H} \times \mathbb{H}$ and, as a consequence, $z = 0$. Thus, $\Pi^{*}_{\flat}$ is injective on $\mathbb{L}^{\bot}$.
    
    Since $\mathbb{L}$ is isotropic, we have $J \mathbb{L} \subset \mathbb{L}^{\bot}$. From this and the above proved we get that $J \mathbb{L} = \mathbb{L}^{\bot}$. The proof is finished.
\end{proof}

Of course, any Lagrange subspace $\mathbb{L}$ satisfies the hypotheses of Theorem \ref{TH: LagrangeSubspaceCriterion} with $\mathbb{H}_{\sharp} = \mathbb{L}$ and $\mathbb{H}_{\flat} = J\mathbb{L}$.

In the following theorem, in a sense converse problem is treated. Namely, it establishes a criterion for a given Lagrange subspace to be a graph over orthogonal decomposition into Lagrange subspaces. Note that it can be reformulated as a criterion for $\mathbb{L}_{0}$ to belong to the chart $\mathcal{O}_{\mathbb{L}}$ associated with $\mathbb{L}$ on the Lagrangian Grassmanian (see \eqref{EQ: LangrGrassCharSAOp}).
\begin{theorem}
	\label{TH: CriterionLagrangeSubspaceIsAGraph}
	Suppose $\mathbb{H} \times \mathbb{H} = \mathbb{L} \oplus \mathbb{L}^{\bot}$ with $\mathbb{L}$ being a Lagrange subspace. Let $\mathbb{L}_{0}$ be another Lagrange subspace and $\Pi_{\mathbb{L}}$ be the orthogonal projector (i.e. along $\mathbb{L}^{\bot}$) onto $\mathbb{L}$. Then the following statements are equivalent:
	\begin{enumerate}
		\item[1.] $\Pi_{\mathbb{L}} \colon \mathbb{L}_{0} \to \mathbb{L}$ is bijective;
		\item[2.] $\Pi_{\mathbb{L}} \colon \mathbb{L}_{0} \to \mathbb{L}$ is injective with closed range;
		\item[3.] There exists a closed subspace $\mathbb{L}_{c}$ such that $\mathbb{L}_{0} \oplus \mathbb{L}_{c} \oplus \mathbb{L}^{\bot} = \mathbb{H} \times \mathbb{H}$.
	\end{enumerate}
\end{theorem}
\begin{proof}
	Since item 1 clearly implies validity of the other items, it is sufficient to show the implications $3 \Rightarrow 2 \Rightarrow 1$.
	
	For $3 \Rightarrow 2$, we apply Lemma \ref{LEM: DirectSumDecompositionCorrective} to get that $\Pi_{\mathbb{L}} \colon \mathbb{L}_{0} \oplus \mathbb{L}_{c} \to \mathbb{L}$ is a bijection. In particular, $\Pi_{\mathbb{L}} \colon \mathbb{L}_{0} \to \mathbb{L}$ is injection and, by the Bounded Inverse Theorem, the range $\Pi_{\mathbb{L}}\mathbb{L}_{0}$ is closed. 
	
	For $2 \Rightarrow 1$, we assume that $\Pi_{\mathbb{L}} \mathbb{L}_{0} \not= \mathbb{L}$ (otherwise we are done). Then there exists a non-zero element $e \in \mathbb{L}$ which is orthogonal to $\Pi_{\mathbb{L}} \mathbb{L}_{0}$ and, consequently, it must be also orthogonal to $\mathbb{L}_{0}$. Since $\mathbb{L}_{0}$ is Lagrange, we have $e \in \mathbb{L}^{\bot}_{0} = J \mathbb{L}_{0}$ and, as a consequence, $J e \in \mathbb{L}_{0} \cap \mathbb{L}^{\bot}$ since $\mathbb{L}$ is Lagrange. Thus $\Pi_{\mathbb{L}}Je = 0$ that contradicts the injectivity. The proof is finished.
\end{proof}

\begin{remark}
	Let us illustrate that the closedness of $\Pi_{\mathbb{L}}$ from item 2 of Theorem \ref{TH: CriterionLagrangeSubspaceIsAGraph} cannot be omitted. For this, we will construct a Lagrange subspace $\mathbb{L}_{0}$ such that $\Pi_{\mathbb{L}} \colon \mathbb{L}_{0} \to \mathbb{L}$ is injective, but not surjective, with dense range. Let us fix an orthonormal basis $\{ e_{k}\}_{k \geq 1}$ in $\mathbb{L}$. Since $\mathbb{L}$ is Lagrange, $\{ Je_{k}\}_{k \geq 1}$ is an orthonormal basis for $\mathbb{L}^{\bot} = J \mathbb{L}$. Let $\mathbb{L}_{0}$ admit the orthonormal basis of linear combinations $a_{k} e_{k} + b_{k} J e_{k}$, where $k=1,2,\ldots$, with $|a_{k}|^{2} + |b_{k}|^{2} = 1$ and $a_{k} > 0$ such that $a_{k} \to 0$ as $k \to \infty$. Clearly, $\Pi_{\mathbb{L}}$ maps $\mathbb{L}_{0}$ injectively and densely into $\mathbb{L}$, but it is not continuously invertible on the range since $\Pi_{\mathbb{L}}( a_{k} e_{k} + b_{k} J e_{k} ) =  a_{k} e_{k}$ can be arbitrarily small as $k \to \infty$.
	
	Note also that the situation of $\Pi_{\mathbb{L}} \colon \mathbb{L}_{0} \to \mathbb{L}$ being injective with nondense range is impossible by the argument used in the implication $2 \Rightarrow 1$ from Theorem \ref{TH: CriterionLagrangeSubspaceIsAGraph} with $e$ taken from the closure of $\Pi_{\mathbb{L}}\mathbb{L}_{0}$. \qed
\end{remark}

For the following proposition we fix a complete metric space $\mathcal{Q}$.
\begin{proposition}
	\label{PROP: LagrangeSubspacesContinuousDependence}
	For each $q \in \mathcal{Q}$, let $\mathbb{L}(q)$ a Lagrange subspace in $\mathbb{H} \times \mathbb{H}$. Suppose there exists a decomposition $\mathbb{H} \times \mathbb{H} = \mathbb{H}_{\sharp} \oplus \mathbb{H}_{\flat}$ and a family of bounded linear operators $M(q) \colon \mathbb{H}_{\sharp} \to \mathbb{H}_{\flat}$ such that for any $q \in \mathcal{Q}$ we have
	\begin{equation}
		\mathbb{L}(q) = \{ z_{\sharp} + M(q)z_{\sharp} \ | \ z_{\sharp} \in \mathbb{H}_{\sharp} \}.
	\end{equation}
	Then the following properties are equivalent:
	\begin{enumerate}
		\item[1)] The mapping $\mathcal{Q} \ni q \mapsto \mathbb{L}(q) \in \Lambda(\mathbb{H})$ is continuous;
		\item[2)] The mapping $\mathcal{Q} \ni q \mapsto M(q) \in \mathcal{L}(\mathbb{H}_{\sharp};\mathbb{H}_{\flat})$ is continuous in the operator norm.
	\end{enumerate}
\end{proposition}
\begin{proof}
	$1) \Rightarrow 2)$: let us fix $q_{0} \in \mathcal{Q}$ and define  $T(q)z_{\sharp} := \Pi_{\mathbb{L}(q)}( z_{\sharp} + M(q_{0})z_{\sharp}  )$ for any $z_{\sharp} \in \mathbb{H}_{\sharp}$ and $q \in \mathcal{Q}$, where $\Pi_{\mathbb{L}(q)}$ is the orthogonal projector onto $\mathbb{L}(q)$. From \eqref{EQ: MetricLagrangianGrassmanian} we have that $T(q) \to T(q_{0})$ in the operator norm of $\mathcal{L}(\mathbb{H}_{\sharp};\mathbb{H} \times \mathbb{H})$ as $q \to q_{0}$. From this, we also have the convergence $\Pi_{\sharp}T(q) \to \Pi_{\sharp}T(q_{0}) = \operatorname{Id}_{\mathbb{H}_{\sharp}}$ in the operator norm  from $\mathcal{L}(\mathbb{H}_{\sharp})$ as $q \to q_{0}$, where $\Pi_{\sharp}$ is the projector onto $\mathbb{H}_{\sharp}$ along $\mathbb{H}_{\flat}$. From this, the inverse $\Pi^{-1}_{q}$ to $\Pi_{\sharp}T(q)$ is well-defined in a neighborhood of $q_{0}$ and $\Pi^{-1}_{q} \to \Pi^{-1}_{q_{0}} = \operatorname{Id}_{\mathbb{H}_{\sharp}}$ as $q \to q_{0}$. Clearly,
	\begin{equation}
		z_{\sharp} + M(q) z_{\sharp} = T(q)\Pi^{-1}_{q}z_{\sharp}.
	\end{equation}
	for any $z_{\sharp} \in \mathbb{H}_{\sharp}$. In particular, $M(q) = \Pi_{\flat} T(q) \Pi^{-1}_{q}$ tends to $\Pi_{\flat}T(q_{0}) \Pi^{-1}_{q_{0}} = M(q_{0})$ as $q \to q_{0}$ in the operator norm.
	
	$2) \Rightarrow 1)$: fix $q_{0} \in \mathcal{Q}$. Since for any $q \in \mathcal{Q}$ the projector $\Pi_{\mathbb{L}(q)}$ is orthogonal, for any $z \in \mathbb{H} \times \mathbb{H}$ the vector $\Pi_{\mathbb{L}(q)}z$ is the unique element of $\mathbb{L}(q)$ which minimizes the difference $|z-\Pi_{\mathbb{L}(q)}z|_{\mathbb{H} \times \mathbb{H}}$. By our assumptions,
	$\Pi_{\mathbb{L}(q)}z = \Pi_{\sharp}\Pi_{\mathbb{L}(q)}z + M(q) \Pi_{\sharp}\Pi_{\mathbb{L}(q)}z$ uniformly in $|z|_{\mathbb{H} \times \mathbb{H}} \leq 1$ approximates $\Pi_{\sharp}\Pi_{\mathbb{L}(q)}z + M(q_{0})\Pi_{\sharp}\Pi_{\mathbb{L}(q)}z \in \mathbb{L}(q_{0})$ as $q \to q_{0}$. Conversely, the vector $\Pi_{\mathbb{L}(q_{0})} z$ can be approximated by $\Pi_{\sharp}\Pi_{\mathbb{L}(q_{0})}z + M(q) \Pi_{\sharp}\Pi_{\mathbb{L}(q_{0})}z \in \mathbb{L}(q)$ uniformly in $|z|_{\mathbb{H} \times \mathbb{H}} \leq 1$ as $q \to q_{0}$. Along with the previous, this implies that $\Pi_{\mathbb{L}(q)} \to \Pi_{\mathbb{L}(q_{0})}$ in the operator norm or, by \eqref{EQ: MetricLagrangianGrassmanian}, $\mathbb{L}(q) \to \mathbb{L}(q_{0})$ in $\Lambda(\mathbb{H})$ as $q \to q_{0}$. The proof is finished.
\end{proof}
\subsection{Lyapunov-Perron operators for exponential dichotomies of dual stationary systems}
\label{SUBSEC: LyapPerronStatDich}
Let $\breve{\mathbb{H}}$ be a Hilbert space that splits into a direct sum $\breve{\mathbb{H}} = \mathbb{H}_{\sharp} \oplus \mathbb{H}_{\flat}$ of two its closed subspaces $\mathbb{H}_{\sharp}$ and $\mathbb{H}_{\flat}$. Suppose we are also given with two exponentially stable $C^{0}$-semigroups\footnote{For an introduction to the theory of $C_{0}$-semigroups we refer to the monographs of K.-J.~Engel and R.~Nagel \cite{EngelNagel2000} and S.G.~Krein \cite{Krein1971}.} $G_{\sharp}$  and $G_{\flat}$ in $\mathbb{H}_{\sharp}$ and $\mathbb{H}_{\flat}$ respectively. Consider the generators $A_{\sharp} \colon \mathcal{D}(A_{\sharp}) \subset \mathbb{H}_{\sharp} \to \mathbb{H}_{\sharp}$ and $A_{\flat} \colon \mathcal{D}(A_{\flat}) \subset \mathbb{H}_{\flat} \to \mathbb{H}_{\flat}$ of $G_{\sharp}$ and $G_{\flat}$ respectively. Let $\Pi_{\sharp}$ and $\Pi_{\flat}$ be the complementary projectors associated with the decomposition $\breve{\mathbb{H}} = \mathbb{H}_{\sharp} \oplus \mathbb{H}_{\flat}$, i.e. $\operatorname{Ran}\Pi_{\sharp} = \operatorname{Ker}\Pi_{\flat} = \mathbb{H}_{\sharp}$ and  $\operatorname{Ker}\Pi_{\sharp} = \operatorname{Ran}\Pi_{\flat} = \mathbb{H}_{\flat}$. We define an operator $\breve{A}$ in $\breve{\mathbb{H}}$ by putting $\breve{A}v := A_{\sharp}\Pi_{\sharp}v - A_{\flat}\Pi_{\flat}v$ for all $v$ from $\mathcal{D}(\breve{A}) :=  \mathcal{D}(A_{\sharp}) + \mathcal{D}(A_{\flat})$.

We are interested in the study of solutions on $\mathbb{R}$ to the linear inhomogeneous problem
\begin{equation}
	\label{EQ: LyapunovPerronInhom}
	\dot{z}(t) = \breve{A}z(t) + f(t),
\end{equation}
where $f(\cdot) \in L_{2}(\mathbb{R};\breve{\mathbb{H}})$. Here solutions should be understood in terms of the decomposition $z(t) = z_{\sharp}(t) + z_{\flat}(t)$, where $z_{\sharp}(t) := \Pi_{\sharp}z(t)$ and $z_{\flat}(t) = \Pi_{\flat}z(t)$ should satisfy as mild solutions\footnote{Recall that for $A$ being the generator of a $C_{0}$-semigroup in $\breve{\mathbb{H}}$ and $f(\cdot) \in L_{2}(0,T;\breve{\mathbb{H}})$ a \textit{mild solution} to $\dot{v}(t) = Av(t) + f(t)$ on $[0,T]$ through $v_{0} \in \breve{\mathbb{H}}$ is given by the Cauchy formula
	\begin{equation}
		\label{EQ: CauchyFormulaInhomogeneous}
		v(t) = G(t)v_{0} + \int_{0}^{t}G(t-s)f(s)ds \text{ for } t \in [0,T].
	\end{equation}
} the equations
\begin{equation}
	\begin{split}
		\dot{z}_{\sharp}(t) &= A_{\sharp} z_{\sharp}(t) + \Pi_{\sharp}f(t), \\
		\dot{z}_{\flat}(t) &= -A_{\flat} z_{\flat}(t) + \Pi_{\flat}f(t).
	\end{split}
\end{equation}
By our assumptions, the first equation is generally solvable in the future only. Contrary, the second equation is generally solvable only in the past. For $f \equiv 0$, the problem admits an exponentially dichotomy, namely, solutions for $z_{0} \in \mathbb{H}_{\sharp}$ are given by $G_{\sharp}(t)z_{0}$ for $t \geq 0$ and exponentially decay as $t \to +\infty$. On the other hand, solutions for $z_{0} \in \mathbb{H}_{\flat}$ are given by $G_{\flat}(-t)z_{0}$ for $t \leq 0$ and exponentially decay as $t \to -\infty$. In other words, there exist $\varepsilon>0$ and $M>0$ such that
\begin{equation}
	\label{EQ: ExponentialDecayStationaryEstimate}
	\begin{split}
		| G_{\sharp}(t)z_{0} |_{\breve{\mathbb{H}}} &\leq M e^{-\varepsilon t} |z_{0}| \text{ for any } z_{0} \in \mathbb{H}_{\sharp} \text{ and } t \geq 0,\\
		| G_{\flat}(-t)z_{0} |_{\breve{\mathbb{H}}} &\leq M e^{\varepsilon t} |z_{0}| \text{ for any } z_{0} \in \mathbb{H}_{\flat} \text{ and } t \leq 0.
	\end{split}
\end{equation}

From this it is not hard to see validity of the following lemma.
\begin{lemma}
	\label{LEM: UniqueZeroSolutionStatDichotomies}
	For $f \equiv 0$, the zero solution is the only solution to \eqref{EQ: LyapunovPerronInhom} that belongs to $L_{2}(\mathbb{R};\breve{\mathbb{H}})$ or $C_{b}(\mathbb{R};\breve{\mathbb{H}})$.
\end{lemma}
\begin{proof}
	Suppose $z(\cdot)$ is a solution on $\mathbb{R}$ belonging to any of the classes. Then we have for any $s \in \mathbb{R}$ that
	\begin{equation}
		\label{EQ: LemmaUniqueZeroSolutionStationaryDich}
		\begin{split}
			z_{\sharp}(s) = G_{\sharp}(t)z_{\sharp}(-t+s) \text{ for any } t \geq 0,\\
			z_{\flat}(s) = G_{\flat}(-t)z_{\flat}(-t+s) \text{ for any } t \leq 0.
		\end{split}
	\end{equation}
	
	In the case $z(\cdot) \in C_{b}(\mathbb{R};\breve{\mathbb{H}})$, we pass to the limit in \eqref{EQ: LemmaUniqueZeroSolutionStationaryDich} as $t \to +\infty$ or $t \to -\infty$ and \eqref{EQ: ExponentialDecayStationaryEstimate} gives that $z(s) = 0$ for any $s \in \mathbb{R}$.
	
	In the case $z(\cdot) \in L_{2}(\mathbb{R};\breve{\mathbb{H}})$, we use \eqref{EQ: LemmaUniqueZeroSolutionStationaryDich}, \eqref{EQ: ExponentialDecayStationaryEstimate} and the H\"{o}lder inequality to get (here $L_{2}$ stands for $L_{2}(\mathbb{R};\breve{\mathbb{H}})$)
	\begin{equation}
		\begin{split}
			\|z_{\sharp}(\cdot)\|_{L_{2}} = \left(\int_{-\infty}^{+\infty}|G_{\sharp}(t)z_{\sharp}(-t+s)|^{2}_{\breve{\mathbb{H}}}ds\right)^{1/2} \leq \\ \leq Me^{-\varepsilon t} \left(\int_{-\infty}^{+\infty}|z_{\sharp}(-t+s)|^{2}_{\breve{\mathbb{H}}}ds\right)^{1/2} = Me^{-\varepsilon t} \|z_{\sharp}(\cdot)\|_{L_{2}}
		\end{split}
	\end{equation}
	for any $t \geq 0$. Taking it to the limit as $t \to +\infty$, we obtain $z_{\sharp}(\cdot) \equiv 0$. Analogously, one may show that $z_{\flat}(\cdot) \equiv 0$. The proof is finished.
\end{proof}

\begin{theorem}
	\label{TH: LyapunovPerronOperatorStationaryTheorem}
	In the above context, for any $f(\cdot) \in L_{2}(\mathbb{R};\breve{\mathbb{H}})$ (resp. $f(\cdot) \in C_{b}(\mathbb{R};\breve{\mathbb{H}})$) there exists a unique mild solution $z(\cdot) \in L_{2}(\mathbb{R};\breve{\mathbb{H}})$ (resp. $z(\cdot) \in C_{b}(\mathbb{R};\breve{\mathbb{H}})$) to \eqref{EQ: LyapunovPerronInhom} and it also belongs to $C_{b}(\mathbb{R};\breve{\mathbb{H}})$. It is given by
	\begin{equation}
		\label{EQ: UniqueDichSolutionFormula}
		z(t) = \int_{-\infty}^{+\infty}F_{\breve{A}}(t,s)f(s)ds =: \left(\mathfrak{L}_{\breve{A}}f\right)(t) \text{ for } t \in \mathbb{R},
	\end{equation}
	where the operator kernel (sometimes called the Green function) $F_{\breve{A}}(t,s) \in \mathcal{L}(\breve{\mathbb{H}})$ is given by
	\begin{equation}
		\label{EQ: KernelLyapPerronOperatorDescription}
		F_{\breve{A}}(t,s) = \begin{cases}
			G_{\sharp}(t-s)\Pi_{\sharp}, &\text{ for } t > s,\\
			-G_{\flat}(s-t)\Pi_{\flat}, &\text{ for } t < s.
		\end{cases}
	\end{equation}
	Moreover, the operator $\mathfrak{L}_{\breve{A}}$ which takes $f(\cdot)$ to such $z(\cdot)$ is bounded as an operator in $L_{2}(\mathbb{R};\breve{\mathbb{H}})$, $C_{b}(\mathbb{R};\breve{\mathbb{H}})$ or from $L_{2}(\mathbb{R};\breve{\mathbb{H}})$ to $C_{b}(\mathbb{R};\breve{\mathbb{H}})$. In particular,
	\begin{equation}
		\label{EQ: LyapPerronStatOperatorL2NormEstimate}
		\|\Pi_{\sharp} z(\cdot) \|_{L_{2}(\mathbb{R};\breve{\mathbb{H}})} \leq \frac{M}{\varepsilon} \text{ and } \|\Pi_{\sharp} z(\cdot) \|_{L_{2}(\mathbb{R};\breve{\mathbb{H}})} \leq \frac{M}{\varepsilon},
	\end{equation}
	where $M$ anb $\varepsilon$ are given by \eqref{EQ: ExponentialDecayStationaryEstimate}.
\end{theorem}
\begin{proof}
	From Lemma \ref{LEM: UniqueZeroSolutionStatDichotomies} we immediately get the uniqueness.
	
	To see how the formula \eqref{EQ: UniqueDichSolutionFormula} can be obtained, let us assume that the required solution $z(\cdot)$ exists. Then we use the following inhomogeneous analog of \eqref{EQ: LemmaUniqueZeroSolutionStationaryDich} derived from the Cauchy formula \eqref{EQ: CauchyFormulaInhomogeneous} as
	\begin{equation}
		\label{EQ: InhomogeneousCauchyFormulaStationaryDich}
		\begin{split}
			z_{\sharp}(s) = G_{\sharp}(t)z_{\sharp}(-t+s) + \int_{-t+s}^{s}G_{\sharp}(s-\theta)\Pi_{\sharp}f(\theta)d\theta \text{ for any } t \geq 0,\\
			z_{\flat}(s) = G_{\flat}(-t)z_{\flat}(-t+s) - \int_{s}^{-t+s}G_{\flat}(\theta-s)\Pi_{\flat}f(\theta)d\theta \text{ for any } t \leq 0.
		\end{split}
	\end{equation}
	Now we pass to the limit as $t \to +\infty$ (resp. $t \to -\infty$) in the first (resp. second) identity from \eqref{EQ: InhomogeneousCauchyFormulaStationaryDich}. In the case $z(\cdot) \in L_{2}(\mathbb{R};\breve{\mathbb{H}})$ we do it for a sequence $t=t_{k}$, where $k=1,2,\ldots$, for which $z(t_{k})$ is uniformly bounded. Since the first terms decay to $0$ due to \eqref{EQ: ExponentialDecayStationaryEstimate}, this gives \eqref{EQ: UniqueDichSolutionFormula}. 
	
	It is not hard to see that \eqref{EQ: UniqueDichSolutionFormula} defines a bounded operator in any of the mentioned senses. Moreover, it provides indeed a solution to \eqref{EQ: LyapunovPerronInhom}. Let us illustrate this by means of \eqref{EQ: LyapPerronStatOperatorL2NormEstimate}. Indeed, by \eqref{EQ: KernelLyapPerronOperatorDescription}, \eqref{EQ: ExponentialDecayStationaryEstimate} and the Cauchy inequality, we have
	\begin{equation}
		\begin{split}
			\vline\int_{-\infty}^{+\infty}\Pi_{\sharp}F_{\breve{A}}(t,s)f(s)ds \vline &\leq M \int_{-\infty}^{t}e^{-\varepsilon(t-s)}|\Pi_{\sharp}f(s)|_{\mathbb{H}}ds \leq \\ &\leq
			\frac{M}{\sqrt{\varepsilon}} \left( \int_{-\infty}^{t}e^{-\varepsilon(t-s)}|\Pi_{\sharp}f(s)|^{2}_{\mathbb{H}}ds  \right)^{1/2} 
		\end{split}
	\end{equation}
	Consequently, from the Fubini theorem we have
	\begin{equation}
		\begin{split}
			\int_{-\infty}^{+\infty}|\Pi_{\sharp}z(t) |^{2}_{\mathbb{H}} dt \leq \frac{M^{2}}{\varepsilon}\int_{-\infty}^{+\infty}\int_{-\infty}^{t}e^{-\varepsilon(t-s)}|\Pi_{\sharp}f(s)|^{2}_{\mathbb{H}}ds dt = \\ =
			\frac{M^{2}}{\varepsilon}\int_{-\infty}^{+\infty}|\Pi_{\sharp}f(s)|^{2}_{\mathbb{H}}ds\int_{-\infty}^{t}e^{-\varepsilon(t-s)} dt = \frac{M^{2}}{\varepsilon^{2}} \| \Pi_{\sharp} f(s) \|^{2}_{L_{2}(\mathbb{R};\breve{\mathbb{H}})}.
		\end{split}
	\end{equation}
	Analogously we obtain estimates by applying $\Pi_{\flat}$. The proof is finished. 
\end{proof}

We call the constructed above operator $\mathfrak{L}_{\breve{A}}$ \textit{the Lyapunov-Perron operator} associated with $\breve{A}$. However, under such a name it is usually understood compositions of $\mathfrak{L}_{\breve{A}}$ with certain operators arising in perturbation problems.

Let us establish the following technical lemma.
\begin{lemma}
	\label{LEM: FourierSolutionOnRLemma}
	Let $f(\cdot) \in L_{2}(\mathbb{R};\breve{\mathbb{H}})$ and put $z = \mathfrak{L}_{\breve{A}}f$. Then the Fourier transform $\widehat{z}(\cdot)$ of $z(\cdot)$ for almost all $\omega \in \mathbb{R}$ satisfies $\widehat{z}(\omega) \in \mathcal{D}(\breve{A})$ and
	\begin{equation}
		\label{EQ: FourierTransformBoundedSolution}
		i \omega \widehat{z}(\omega) = \breve{A} \widehat{z}(\omega) + \widehat{f}(\omega).
	\end{equation}
\end{lemma}
\begin{proof}
	Since the Fourier transform is linear and commutes with the operators in $L_{2}(\mathbb{R};\breve{\mathbb{H}})$ given by the pointwise action of $\Pi_{\sharp}$ or $\Pi_{\flat}$, it is sufficient to show analogs of \eqref{EQ: FourierTransformBoundedSolution} for the components $z_{\sharp}(\cdot) = \Pi_{\sharp} (\cdot)z$ and $z_{\flat}(\cdot) = \Pi_{\flat}z(\cdot)$.
	
	Let $T>0$ be fixed. We are going to approximate $z_{\sharp}(\cdot)$ on $[-T,T]$ by smooth solutions. For $k=1,2,\ldots$, there exist $v_{k}(-T) \in \mathcal{D}(A_{\sharp})$ tending to $z_{\sharp}(-T)$ in $\mathbb{H}$ and $g_{k}(\cdot) \in C^{1}([-T,T];\mathbb{H})$ tending to $\Pi_{\sharp}
	f(\cdot)$ in $L_{2}(-T,T;\mathbb{H})$ as $k \to \infty$. By Theorem 6.5 from Chapter I in \cite{Krein1971}, the corresponding solution $v_{k}(\cdot)$ of
	\begin{equation}
		\label{EQ: FourierLemmaApprox}
		\dot{v}_{k}(t) = A_{\sharp} v_{k}(t) + g_{k}(t)
	\end{equation}
	on $[-T,T]$ with the given $v_{k}(-T)$ is classical, i.e. $v_{k}(\cdot) \in C^{1}([-T,T];\mathbb{H})$, $v_{k}(\cdot) \in C([-T,T];\mathcal{D}(A_{\sharp}))$, and $v_{k}(t)$ satisfies \eqref{EQ: FourierLemmaApprox} for any $t \in [-T,T]$. Integrating by parts, we get
	\begin{equation}
		\int_{-T}^{T}e^{-i\omega t} \dot{v}_{k}(t) = e^{-i\omega T} v_{k}(T) - e^{i \omega T} v_{k}(-T) + i\omega\int_{-T}^{T} e^{-i \omega t}v_{k}(t)dt.
	\end{equation}
	Putting 
	\begin{equation}
		\widehat{v}^{T}_{k}(\omega) := \frac{1}{\sqrt{2\pi}} \int_{-T}^{T}e^{-i\omega t}v_{k}(t)dt \text{ and } \widehat{g}^{T}_{k}(\omega):= \frac{1}{\sqrt{2\pi}} \int_{-T}^{T}e^{-i\omega t}g_{k}(t)dt,
	\end{equation}
	we get for any $\omega \in \mathbb{R}$ that
	\begin{equation}
		\label{EQ: FourierLemmaApproximateEqs}
		i \omega \widehat{v}^{T}_{k}(\omega) = A_{\sharp} \widehat{v}^{T}_{k}(\omega) + \widehat{g}^{T}_{k}(\omega) +  \frac{1}{\sqrt{2\pi}} \left(e^{-i\omega T} v_{k}(T) - e^{i \omega T} v_{k}(-T) \right).
	\end{equation}
	From this it is clear that $\widehat{v}^{T}_{k}(\cdot) \to \widehat{v}^{T}(\cdot)$ in $C([-T,T];\mathcal{D}(A_{\sharp}))$ as $k \to \infty$. Passing to the limit as $k \to \infty$ in \eqref{EQ: FourierLemmaApproximateEqs}, we get
	\begin{equation}
		i \omega \widehat{z}^{T}_{\sharp}(\omega) = A_{\sharp}\widehat{z}^{T}_{\sharp}(\omega) + \Pi_{\sharp}\widehat{f}^{T}(\omega) + \frac{1}{\sqrt{2\pi}} \left(e^{-i\omega T} z_{\sharp}(T) - e^{i \omega T} z_{\sharp}(-T) \right),
	\end{equation}
	where
	\begin{equation}
		\label{EQ: FourierLemmaApproximateEqsZsharp}
		\widehat{z}^{T}_{\sharp}(\omega) := \frac{1}{\sqrt{2\pi}} \int_{-T}^{T}e^{-i\omega t}z_{\sharp}(t)dt \text{ and } \widehat{f}^{T}(\omega):= \frac{1}{\sqrt{2\pi}} \int_{-T}^{T}e^{-i\omega t}f(t)dt.
	\end{equation}
	By continuity of the Fourier transform, $\widehat{z}^{T}_{\sharp}(\cdot)$ tends to $\widehat{z}_{\sharp}(\cdot)$ in $L_{2}(\mathbb{R};\mathbb{H})$ as $T \to +\infty$. From \eqref{EQ: FourierLemmaApproximateEqsZsharp} we also have that the convergence holds in $L_{2, loc}(\mathbb{R};\mathcal{D}(A_{\sharp}))$. Let us for any $k=1,2,\ldots$ choose $T=T_{k}$ such that $z_{\sharp}(T_{k})$ and $z_{\sharp}(-T_{k})$ tends to $0$ and $\widehat{z}^{T}_{\sharp}(\omega)$ tends to $\widehat{z}_{\sharp}(\omega)$ as $k \to \infty$ for almost all $\omega \in \mathbb{R}$. Passing in \eqref{EQ: FourierLemmaApproximateEqsZsharp} with $T=T_{k}$ to the limit  as $k \to \infty$ gives the required relation for $\widehat{z}_{\sharp}(\cdot)$. For $\widehat{z}_{\flat}$, one may apply similar arguments. The proof is finished.
\end{proof}

Now we will describe our main example to which the above construction will be applied. For this, let $\mathbb{H}$ be a Hilbert space and $A \colon \mathcal{D}(A) \subset \mathbb{H} \to \mathbb{H}$ be the generator of a $C_{0}$-semigroup $G$ which admits an exponentially dichotomy of finite rank $j \geq 0$, i.e. there exists a decomposition $\mathbb{H} = \mathbb{H}^{s}_{A} \oplus \mathbb{H}^{u}_{A}$ with $\dim \mathbb{H}^{u}_{A} = j$ and the corresponding complementary spectral projectors $\Pi^{s}_{A}$ and $\Pi^{u}_{A}$ onto $\mathbb{H}^{s}_{A}$ and $\mathbb{H}^{u}_{A}$ respectively such that
\begin{equation}
	\label{EQ: SpectralProjectorsInvarianceStationary}
	\Pi^{s}_{A}G(t)=G(t)\Pi^{s}_{A} \text{ and } \Pi^{u}_{A}G(t) = G(t) \Pi^{u}_{A} \text{ for any } t \geq 0
\end{equation}
or, in other words, the spaces $\mathbb{H}^{s}_{A}$ and $\mathbb{H}^{u}_{A}$ are invariant w.r.t. $G$, and the $C_{0}$-semigroups $G_{\sharp}$ and $G_{\flat}$ on $\mathbb{H}_{\sharp} := \mathbb{H}^{s}_{A}$ and $\mathbb{H}_{\flat} := \mathbb{H}^{u}_{A}$ given by $G_{\sharp}(t)z_{0} := G(t)z_{0}$ and $G_{\flat}(t)z_{0}:=G(-t)z_{0}$ for $t \geq 0$ and appropriate $z_{0}$ admit an estimate as in \eqref{EQ: ExponentialDecayStationaryEstimate}. Here $G(-t)$ makes sense as the inverse of $G(t)$ restricted to the finite-dimensional subspace $\mathbb{H}^{u}_{A}$.

It is well-known that the adjoint $A^{*}$ of $A$ is the generator of a $C_{0}$-semigroup $G^{*}$ in $\mathbb{H}$ given by the adjoints $G^{*}(t)$ of $G(t)$ for any $t \geq 0$. By taking the adjoints in \eqref{EQ: SpectralProjectorsInvarianceStationary}, one may see that $G^{*}$ also admits an exponential dichotomy of rank $j$ with the spectral projectors $\Pi^{s}_{A^{*}}$ and $\Pi^{u}_{A^{*}}$ given by the adjoints of $\Pi^{s}_{A}$ and $\Pi^{u}_{A}$. We also put $\mathbb{H}^{s}_{A^{*}} := \operatorname{Ran}\Pi^{s}_{A^{*}}$ and $\mathbb{H}^{u}_{A^{*}} := \operatorname{Ran}\Pi^{u}_{A^{*}}$.

We define an operator $\breve{A}$ on $\breve{\mathbb{H}} := \mathbb{H} \times \mathbb{H}$ by
\begin{equation}
	\label{EQ: OperatorABreveFromADef}
	\breve{A} (v,\eta) := Av - A^{*}\eta \text{ for all } (v,\eta) \in \mathcal{D}(\breve{A}):=\mathcal{D}(A) \times \mathcal{D}(A^{*}).
\end{equation}

Note that this operator satisfies the assumptions posed above \eqref{EQ: LyapunovPerronInhom} with $\mathbb{H}_{\sharp} := \{ (v,\eta) \in \mathbb{H} \times \mathbb{H} \ | \ v \in \mathbb{H}^{s}_{A}, \eta \in \mathbb{H}^{u}_{A^{*}} \}$ and $\mathbb{H}_{\flat}:= \{ (v,\eta) \in \mathbb{H} \times \mathbb{H} \ | \ v \in \mathbb{H}^{u}_{A}, \eta \in \mathbb{H}^{s}_{A^{*}} \}$.

Let us rewrite the problem \eqref{EQ: LyapunovPerronInhom} in terms of $z=(v,\eta)$ as
\begin{equation}
	\label{EQ: StatPairedSystemNonAutonomousProblem}
	\begin{split}
		\dot{v}(t) &= Av(t) + f(t),\\
		\dot{\eta}(t) &= -A^{*}\eta(t) + g(t),
	\end{split}
\end{equation}
where $f(\cdot), g(\cdot) \in L_{2}(\mathbb{R};\mathbb{H})$. Recall that any single problem (i.e. w.r.t. $v$ or $\eta$) as well as the coupled problem admits an exponential dichotomy. It is clear that in terms of Theorem \ref{TH: LyapunovPerronOperatorStationaryTheorem} we have
\begin{equation}
	\mathfrak{L}_{\breve{A}}(f,g) = ( \mathfrak{L}_{A}f, \mathfrak{L}_{-A^{*}}g ).
\end{equation}
Let $F_{A}$ and $F_{-A^{*}}$ be the kernels of $\mathfrak{L}_{A}$ and $\mathfrak{L}_{-A^{*}}$ respectively. We have the following proposition.
\begin{proposition}
	In the above context, $F_{-A^{*}}(t,s)$ is the adjoint of $F^{*}_{A}(s,t)$ in $\mathbb{H}$ for any $t \not= s$. In particular, the operator $\mathfrak{L}_{-A^{*}}$ is the adjoint of $\mathfrak{L}_{A}$ in $L_{2}(\mathbb{R};\mathbb{H})$.
\end{proposition}
\begin{proof}
	The first statement follows directly from \eqref{EQ: KernelLyapPerronOperatorDescription} and the second statement follows from straightforward computations of the adjoint of the integral operator $\mathfrak{L}_{A}$ by utilizing the Fubini theorem. The proof is finished.
\end{proof}
\section{Stationary case: the Frequency Condition}
 \label{SEC: LagrangeBundlesStationaryCase}
In this section, we are going to discuss the presence of nonoscillating Lagrange subspaces for stationary Hamiltonian systems in $\breve{\mathbb{H}} := \mathbb{H} \times \mathbb{H}$ given by
\begin{equation}
	\label{EQ: StatHamiltonianSystem}
	\begin{pmatrix}
		\dot{v}(t)\\
		\dot{\eta}(t)
	\end{pmatrix} = H
	\begin{pmatrix}
		v(t)\\
		\eta(t)	
	\end{pmatrix},
\end{equation}
where for a real Hilbert space $\mathbb{U}$ and operators $B \in \mathcal{L}(\mathbb{U};\mathbb{H})$, $F_{1} \in \mathcal{L}(\mathbb{H})$, $F_{2} \in \mathcal{L}(\mathbb{H};\mathbb{U})$ and $F_{3} \in \mathcal{L}(\mathbb{H})$ with $F_{1}$, $F_{3}$ being self-adjoint and $F_{3}$ being invertible and positive-definite\footnote{In fact, the assumptions on the invertibility of $F_{3}$ and positive-definiteness are automatically satisfied under the Frequency Condition.} (i.e. $F_{3} \geq \delta I$ for some $\delta>0$), we have
\begin{equation}
	\label{EQ: HamiltonianStationary}
	H := \begin{pmatrix}
		\widehat{A} & B F^{-1}_{3} B^{*}\\
		F_{1} - F^{*}_{2} F^{-1}_{3}F_{2} & -\widehat{A}^{*}
	\end{pmatrix} = \begin{pmatrix}
	H_{1} \ H_{3}\\
	H_{2} \ -H^{*}_{1}
	\end{pmatrix},
\end{equation}
where $\widehat{A} = A - B F^{-1}_{3} F_{2}$ and $A$ as in \eqref{EQ: OperatorABreveFromADef}, i.e. 

\begin{description}[before=\let\makelabel\descriptionlabel]
	\item[\textbf{(DICH)}\refstepcounter{desccount}\label{DESC: DICHSTATIONARY}] $A$ generates a $C_{0}$-semigroup in $\mathbb{H}$ admitting an exponential dichotomy of finite\footnote{For the statement of Theorem \ref{TH: DichotomyStatOptimizationLP}, this requirement of finiteness can be formally relaxed. Namely, it is used in the proof only to obtain uniform bounds on the resolvent of $A$. However, obtaining such properties (including the exponential dichotomy) in known examples is always related to some kind of asymptotic compactness of the semigroup and the possibility of distinguishing a finite part of the spectrum (see \cite{Anikushin2020FreqDelay}). Moreover, the finiteness of $j$ is essential for further investigations concerned with nonoscillation (see Corollary \ref{COR: StatDichLagrangeSubsFredholmPair} and below).} rank $j \geq 0$.
\end{description}

Let $\langle \cdot, \cdot \rangle_{\breve{\mathbb{H}}}$ denote the inner product in $\breve{\mathbb{H}}$ and recall the complex structure $J$ given by $J(v,\eta) := (-\eta,v)$ for any $(v,\eta) \in \breve{\mathbb{H}}$. Clearly,
\begin{equation}
	\label{EQ: SymplecticIdentityStationary}
	J H + H^{*}J = 0 \text{ on } \mathcal{D}(A) \times \mathcal{D}(A^{*}).
\end{equation}

We understand solutions to \eqref{EQ: StatHamiltonianSystem} as mild solutions to the uncoupled system \eqref{EQ: StatPairedSystemNonAutonomousProblem} with proper $f(\cdot)$ and $g(\cdot)$. Note that we have no a priori knowledge on any structure of solutions to \eqref{EQ: StatHamiltonianSystem}. Thus, it is not clear whether
\begin{equation}
	\label{EQ: SymplecticFormPreservationStationary}
	\langle z_{1}(t), Jz_{2}(t) \rangle_{\breve{\mathbb{H}}} = \langle z_{1}(0), Jz_{2}(0) \rangle_{\breve{\mathbb{H}}} \text{ for any } t \in [0,T].
\end{equation}
holds for any mild solutions $z_{1}(t)=(v_{1}(t),\eta_{1}(t))$ and $z_{2}(t)=(v_{2}(t),\eta_{2}(t))$ to \eqref{EQ: StatHamiltonianSystem} on $[0,T]$ for some $T>0$. Clearly, \eqref{EQ: SymplecticFormPreservationStationary} follows from \eqref{EQ: SymplecticIdentityStationary} under some regularity assumptions. Below, \eqref{EQ: SymplecticFormPreservationStationary} will be proved under the presence of a certain structure which guarantee the existence of classical solutions.

\subsection{Stable Lagrange subspaces}
Let us illustrate how the Lyapunov-Perron operators can be utilized to construct stable Lagrange subspaces for \eqref{EQ: StatHamiltonianSystem}. For this, we rewrite \eqref{EQ: StatHamiltonianSystem} as
\begin{equation}
	\label{EQ: StatHamiltonianSystemRewriten}
	\begin{pmatrix}
		\dot{v}(t)\\
		\dot{\eta}(t)
	\end{pmatrix} =
	\begin{pmatrix}
		Av(t)\\
		-A^{*}\eta(t)	
	\end{pmatrix} + R
	\begin{pmatrix}
		v(t)\\
		\eta(t)	
	\end{pmatrix},
\end{equation}
where the operator matrix $R$ is given by
\begin{equation}
	\label{EQ: StatHamiltonianPerturbationMatrixR}
	R = \begin{pmatrix}
		- B F^{-1}_{3} F_{2} & B F^{-1}_{3} B^{*}\\
		F_{1} - F^{*}_{2} F^{-1}_{3}F_{2} & (B F^{-1}_{3} F_{2})^{*}
	\end{pmatrix}.
\end{equation}

Let us consider a bounded quadratic form $\mathcal{F}(v,\xi)$ of $v \in \mathbb{H}$ and $\xi \in \mathbb{U}$ given by $\mathcal{F}(v,\xi) := (F_{1}v,v)_{\mathbb{H}} + 2(F_{2}v,\xi)_{\mathbb{U}} + (F_{3}\xi,\xi)_{\mathbb{U}}$. A classical frequency condition then requires that 
\begin{equation}
	\label{EQ: FreqConditionStationaryGeneral}
	\mathcal{F}^{\mathbb{C}}(-(A-i\omega I)^{-1}B\xi,\xi) \geq \delta |\xi|^{2}_{\mathbb{U}^{\mathbb{C}}} \text{ for any } \xi \in \mathbb{U}^{\mathbb{C}} \text{ and } \omega \in \mathbb{R}.
\end{equation}
We are going to explore its geometric meaning.

Let us introduce the spaces
\begin{equation}
	\label{EQ: StatDualCouplingStableUnstableSubspaces}
	\begin{split}
		\breve{\mathbb{H}}^{s} := \{ (v,\eta) \in \mathbb{H} \times \mathbb{H} \ | \ v \in \mathbb{H}^{s}_{A}, \eta \in \mathbb{H}^{u}_{A^{*}} \},\\
		\breve{\mathbb{H}}^{u} := \{ (v,\eta) \in \mathbb{H} \times \mathbb{H} \ | \ v \in \mathbb{H}^{u}_{A}, \eta \in \mathbb{H}^{s}_{A^{*}} \}.
	\end{split}
\end{equation}

\begin{proposition}
	\label{PROP: StatDichSubspacesAreLagrange}
	Both subspaces $\breve{\mathbb{H}}^{s}$ and $\breve{\mathbb{H}}^{u}$ from \eqref{EQ: StatDualCouplingStableUnstableSubspaces} are Lagrange.
\end{proposition}
\begin{proof}
	Indeed, for $\breve{\mathbb{H}}^{s}$ we have
	\begin{equation}
		\breve{\mathbb{H}}^{s} = \operatorname{Ran}\Pi \times \{0_{\mathbb{H}}\} \oplus \{0_{\mathbb{H}}\} \times (\operatorname{Ran}\Pi)^{\bot},
	\end{equation}
	where $\Pi = \Pi^{s}_{A}$. Then Proposition \ref{PROP: LagrangeThroughAProjector} gives the desired. For $\breve{\mathbb{H}}^{u}$, the argument is analogous. The proof is finished.
\end{proof}

Let $\mathbb{L}^{+}_{b}$ be the subspace of all $z_{0} \in \breve{\mathbb{H}}$ such that there exists a solution $z(\cdot)=(v(\cdot), \eta(\cdot))$ of \eqref{EQ: StatHamiltonianSystem} on $[0,+\infty)$ with $z(0)=z_{0}$ which is bounded. Moreover, for any $\varepsilon \in \mathbb{R}$ we define a subspace $\mathbb{L}^{+}_{\varepsilon}$ by requiring that $e^{\varepsilon t} z(t)$ should lie in $L_{2}(0,\infty;\breve{\mathbb{H}})$ instead of the boundedness. We have the following theorem.
\begin{theorem}
	\label{TH: DichotomyStatOptimizationLP}
	Consider the Hamiltonian system \eqref{EQ: StatHamiltonianSystem} under \eqref{EQ: FreqConditionStationaryGeneral} and \nameref{DESC: DICHSTATIONARY}. Then $\mathbb{L}^{+}_{b}$ is a Lagrange subspace which coincides with the subspace $\mathbb{L}^{+}_{\varepsilon}$ for all $|\varepsilon| \leq \varepsilon_{0}$ for some $\varepsilon_{0}>0$. Moreover, for any $0<\varepsilon \leq \varepsilon_{0}$ there exists a constant $M_{\varepsilon}>0$ such that the estimate
	\begin{equation}
		\label{EQ: ExpDecayStatHamiltonian}
		|z(t;z_{0})|_{\breve{\mathbb{H}}} \leq M_{\varepsilon} e^{-\varepsilon t} |z_{0}|_{\breve{\mathbb{H}}}
	\end{equation}
	holds for any $z_{0} \in \mathbb{L}^{+}_{b}$ and the corresponding solution $z(t;z_{0})$ with $z(0;z_{0}) = z_{0}$.
	
	In addition, $\mathbb{L}^{+}_{b}$ is given by the graph of a linear bounded operator $M^{+} \colon \breve{\mathbb{H}}^{s} \to \breve{\mathbb{H}}^{u}$ as
	\begin{equation}
		\label{EQ: StatDichSubspaceAsGraph}
		\mathbb{L}^{+}_{b} = \{ z^{s} + M^{+}z^{s} \ | \ z^{s} \in \breve{\mathbb{H}}^{s} \}.
	\end{equation}
\end{theorem}
\begin{proof}
	We are aimed to prove that for any $z^{s} \in \breve{\mathbb{H}}^{s}$ there exists a unique $z_{0}=(v_{0},\eta_{0}) \in \mathbb{L}^{+} := \mathbb{L}_{0}$ such that $\Pi^{s}_{\breve{A}}z_{0} = z^{s}$. To involve Lyapunov-Perron operators into the problem, suppose for a moment that we have already found such $z_{0}$. Then there exists a solution $z(\cdot)=(v(\cdot),\eta(\cdot))$ of \eqref{EQ: StatHamiltonianSystem} on $[0;+\infty)$ with $z(0)=z_{0}$. Define $\Delta z(t) := z(t) - G_{\sharp}(t)z^{s}$. Then $\Pi^{s}_{\breve{A}}\Delta z(0) = 0$. This allows to define a continuous function $w(\cdot)$ on $\mathbb{R}$ as
	\begin{equation}
		\label{EQ: LyapPerronStatDichWDef}
		w(t) = \begin{cases}
			\Delta z(t), \text{ for } t \geq 0,\\
			G_{\flat}(-t)\Pi^{u}_{\breve{A}}\Delta z(0), \text{ for } t < 0.
		\end{cases}
	\end{equation}
	
	Now let us for methodological purposes examine a naive but illustrative way to introduce Lyapunov-Perron operators into the problem. Clearly, $w(\cdot)$ solves 
	\begin{equation}
		\label{EQ: LyapPerronStatDichEquationW}
		\dot{w}(t) = \breve{A}w(t) + R(\mathcal{P}\Delta z)(t) + R g(t; z^{s}),
	\end{equation}
	where $g(t) = g(t; z^{s}) := G_{\sharp}(t)z^{s}$ for $t \geq 0$ and $0$ for $t < 0$ and for any $f(\cdot) \in L_{2}(0,+\infty;\breve{\mathbb{H}})$ and $R$ from \eqref{EQ: StatHamiltonianPerturbationMatrixR} we have
	\begin{equation}
		(\mathcal{P}f)(t) = \begin{cases}
			f(t), \text{ for } t > 0,\\
			0, \text{ for } t \leq 0.
		\end{cases}
	\end{equation}
	By \eqref{EQ: LyapPerronStatDichWDef}, $w(\cdot)$ is a solution to \eqref{EQ: LyapPerronStatDichEquationW} belonging to $L_{2}(\mathbb{R};\breve{\mathbb{H}})$. Since $\mathcal{P}\Delta z + g \in L_{2}(\mathbb{R};\breve{\mathbb{H}})$, Theorem \ref{TH: LyapunovPerronOperatorStationaryTheorem} gives that
	\begin{equation}
		w = \mathfrak{L}_{\breve{A}} R(\mathcal{P} \Delta z + g)
	\end{equation}
	Let $\mathcal{R} \colon L_{2}(\mathbb{R};\breve{\mathbb{H}}) \to L_{2}(0,\infty;\breve{\mathbb{H}})$ be the operator that restricts functions to $(0,\infty)$. Applying $\mathcal{R}$ to the above identity gives
	\begin{equation}
		\label{EQ: StatProblemsLyapPerronNaive}
		\Delta z = \mathfrak{L}_{\breve{A}} R(\mathcal{P} \Delta z + g).
	\end{equation}
	Thus, $\Delta z$ is the fixed point (which is expected to be unique) of the operator from the right-hand side in \eqref{EQ: StatProblemsLyapPerronNaive} considered in $L_{2}(0,\infty;\breve{\mathbb{H}})$.
	
	However, the operator from \eqref{EQ: StatProblemsLyapPerronNaive} is not appropriate from the perspective of control-theoretic nature of the problem. Indeed, from the view of $\mathbb{H}$-valued functions, the $\breve{\mathbb{H}}$-valued function $\Delta z$ has two components and, consequently, it represents two inputs for the operator, but the initial control problem is a problem with a single input. 
	
	Here the right approach is as follows. Let $\mathcal{R}$ take $L_{2}(\mathbb{R};\mathbb{H})$ into $L_{2}(0,\infty;\mathbb{H})$ by restricting functions to $(0,\infty)$ and let $\mathcal{P}$ take $L_{2}(0,\infty;\mathbb{H})$ to $L_{2}(\mathbb{R};\mathbb{H})$ by extending functions by $0$ to the negative semi-axis. Suppose $\Delta z(\cdot) = (\Delta v(\cdot), \Delta \eta(\cdot))$. Let us put $\xi(t) := -F^{-1}_{3}F_{2} \Delta v(t) + F^{-1}_{3}B^{*}\Delta \eta(t)$ for $t \geq 0$ and suppose $Rg(t;z^{s}) = (g_{v}(t), g_{\eta}(t))$ for $t \in \mathbb{R}$. In terms of the operators $\mathfrak{L}_{A}$ and $\mathfrak{L}_{-A^{*}}$ we can characterize $\xi$ as
	\begin{equation}
		\label{EQ: LyapPerronOperatorStationaryTrue}
		\xi =  \mathcal{T}\xi + \mathcal{T}_{0}g,
	\end{equation}
	where
	\begin{equation}
		\begin{split}
			\mathcal{T} &:= -F^{-1}_{3}F_{2}\mathcal{R} \mathfrak{L}_{A} \mathcal{P} B + F^{-1}_{3} B^{*} \mathcal{R} \mathfrak{L}_{-A^{*}} \mathcal{P}\left[ \mathcal{R} F_{1} \mathfrak{L}_{A}\mathcal{P} B + F^{*}_{2} \right],\\
			\mathcal{T}_{0} &:= -F^{-1}_{3}F_{2}\mathcal{R} \mathfrak{L}_{A} g_{v} + F^{-1}_{3} B^{*} \mathcal{R} \mathfrak{L}_{-A^{*}} \mathcal{P} \mathcal{R} F_{1} \mathfrak{L}_{A} g_{\eta}
		\end{split}
	\end{equation}
	
	Let us show that $I - \mathcal{T}$ is invertible in $L_{2}(0,\infty;\mathbb{U}^{\mathbb{C}})$, so \eqref{EQ: LyapPerronOperatorStationaryTrue} has a unique solution. Clearly, for any $\xi(\cdot) \in L_{2}(0,\infty;\mathbb{H})$ we have $\mathcal{T} \xi(\cdot) = -F^{-1}_{3}F_{2} v(\cdot) + F^{-1}_{3}B^{*}\eta(\cdot)$, where $v(\cdot), \eta(\cdot) \in L_{2}(\mathbb{R};\mathbb{H})$ satisfy
	\begin{equation}
		\begin{split}
			\dot{v}(t) &= A v(t) + B\mathcal{P}\xi(t),\\
			\dot{\eta}(t) &= - A^{*}\eta(t) + F_{1} v(t) + F^{*}_{2} \mathcal{P}\xi(t).
		\end{split}
	\end{equation}
	
	We will establish invertibility in the Fourier space. Namely, let $\widehat{v}(\cdot)$, $\widehat{\eta}(\cdot)$ and $\widehat{\xi}(\cdot)$ denote the Fourier transforms of $v(\cdot)$, $\eta(\cdot)$ and $(\mathcal{P} \xi)(\cdot)$ respectively. Then Lemma \ref{LEM: FourierSolutionOnRLemma} gives the relations
	\begin{equation}
		\begin{split}
			\widehat{v}(\omega) &= - (A-i\omega I)^{-1}B \widehat{\xi}(\omega),\\
			\widehat{\eta}(\omega) &= -(-A^{*} -i\omega I)^{-1}( F_{1} \widehat{v}(\omega) + F^{*}_{2}\widehat{\xi}(\omega) )
		\end{split}
	\end{equation}
	for almost all $\omega \in \mathbb{R}$. From this, we get 
	\begin{equation}
		-F^{-1}_{3}F_{2} \widehat{v}(\omega) + F^{-1}_{3}B^{*}\widehat{\eta}(\omega) = M(\omega) \widehat{\xi}(\omega),
	\end{equation}
	where the operator $M(\omega) \in \mathcal{L}(\mathbb{U}^{\mathbb{C}})$ for any $\omega \in \mathbb{R}$ is given by
	\begin{equation}
		\begin{split}
			M(\omega) := F^{-1}_{3} F_{2} (A-i\omega I)^{-1}B +\\+ F^{-1}_{3}B^{*}(-A^{*} - i\omega I)^{-1}\left(F_{1} (A-i\omega I)^{-1}B - F^{*}_{2}\right).
		\end{split}
	\end{equation}
	Note that the operator $F_{3} M(\omega)$ is self-adjoint and for any $v \in \mathcal{D}(A)$, $\xi \in \mathbb{U}^{\mathbb{C}}$ and $\omega \in \mathbb{R}$ related by $i \omega v = A v + B\xi$ from \eqref{EQ: FreqConditionStationaryGeneral} we have
	\begin{equation}
		\label{EQ: ProofStatDichUtilizeFreqIneq}
		\mathcal{F}^{\mathbb{C}}(v, \xi) = ( F_{3}(I - M(\omega))\xi,\xi)_{\mathbb{U}^{\mathbb{C}}} \geq \delta |\xi|^{2}_{\mathbb{U}^{\mathbb{C}}}.
	\end{equation}
	This means that the quadratic form of the self-adjoint operator $F_{3}(I - M(\omega))$ is coercive. Now the Lax-Milgram theorem gives that the operator is invertible and the norm of its inverse can be estimated as $\delta^{-1}$. In particular, $I - M(\omega)$ is invertible and
	\begin{equation}
		\label{EQ: NormEstimateStationaryLP}
		\|(I - M(\omega))^{-1}\| \leq \frac{\|F_{3}\|}{\delta} \text{ for any } \omega \in \mathbb{R}.
	\end{equation}
	
	Since the Fourier transform is a unitary operator, the operator $I-\mathcal{T}$ in $L_{2}(0,\infty;\mathbb{U})$ is invertible and its inverse admits the same estimate as in \eqref{EQ: NormEstimateStationaryLP}. 
	
	Consequently, for any $z^{s} \in \breve{\mathbb{H}}^{s}$ and $g(t) = g(t;z^{s}) = G_{\sharp}(t)z^{s}$ for $t \geq 0$ we may put 
	\begin{equation}
		\begin{split}
			\xi &:=(I-\mathcal{T})^{-1}\mathcal{T}_{0}g,\\
			\Delta v &:= \mathcal{L}_{A}B\mathcal{P}\xi,\\
			\Delta\eta &:= \mathcal{L}_{-A^{*}} (F_{1}\Delta v + F^{*}_{2}\mathcal{P}\xi)
		\end{split}
	\end{equation}
	and $\Delta z = (\Delta v, \Delta \eta)$ is the required difference and $\Pi^{s}_{\breve{A}}\Delta z(0) = 0$. We put $M^{+}z^{s} :=\Delta z (0)$. Note that $\xi$ continuously depend on $z^{s}$ in $L_{2}(0,\infty;\mathbb{U}^{\mathbb{C}})$. Then Theorem \ref{TH: LyapunovPerronOperatorStationaryTheorem} gives that $\Delta z(0)$ continuously depend on $z^{s}$. Consequently, $M^{+}z^{s}$ is bounded.
	
	Thus, we constructed the subspace $\mathbb{L}_{\varepsilon}$ for $\varepsilon=0$. Note that we may apply the same arguments for $\breve{A}$ exchanged with the operator $\breve{A} + \varepsilon I$ with $|\varepsilon| < \varepsilon_{0}$. Indeed, the exponential dichotomy and the frequency condition \eqref{EQ: FreqConditionStationaryGeneral} are also satisfied if $\varepsilon_{0}>0$ is sufficiently small. For the latter, this is due to the First Resolvent Identity, the boundedness of $\mathcal{F}$ and uniform boundedness of the resolvent of $\breve{A}$ on the imaginary axis (see the proof of Theorem 4.2 in \cite{Anikushin2020FreqDelay} for similar arguments). It is clear that $\mathbb{L}^{+}_{\varepsilon}$ is just the space constructed by above arguments for the operator $\breve{A} + \varepsilon I$.
	
	By definition, $\mathbb{L}^{+}_{\varepsilon_{1}} \supset \mathbb{L}^{+}_{\varepsilon_{2}}$ for $\varepsilon_{1} < \varepsilon_{2}$ and $\mathbb{L}^{+}_{\varepsilon} \subset \mathbb{L}^{+}_{b} \subset \mathbb{L}^{+}_{-\varepsilon}$ for $\varepsilon > 0$. Since $\Pi^{s}_{\breve{A}}$ is bijective on each $\mathbb{L}_{\varepsilon}$ by construction, all the spaces must coincide. 
	
	Using the boundedness of the Lyapunov-Perron operator from $L_{2}(\mathbb{R};\breve{\mathbb{H}})$ into $C_{b}(\mathbb{R};\breve{\mathbb{H}})$ given by Theorem \ref{TH: LyapunovPerronOperatorStationaryTheorem} and applied to the problem $\breve{A} + \varepsilon I$ with $|\varepsilon| \leq \varepsilon_{0}$, we get that \eqref{EQ: ExpDecayStatHamiltonian} is satisfied with
	\begin{equation}
		M_{\varepsilon}:=\sup_{t \geq 0}\sup_{|z_{0}|_{\breve{\mathbb{H}}} \leq 1}| e^{\varepsilon t}z(t;q,z_{0}) |_{\breve{\mathbb{H}}} < \infty.
	\end{equation}
	
	It remains to show that $\mathbb{L}^{+}_{b}$ is Lagrange. For this, note that the restriction of the Hamiltonian $H$ to $\mathbb{L}^{+}_{b}$ generates a $C_{0}$-semigroup due to the continuous dependence of $\Delta z(0)$ on $z^{s}$. In particular, for a dense subset of $z_{0} \in \mathbb{L}^{+}_{b}$ there exist classical solutions $z(\cdot)=(v(\cdot),\eta(\cdot))$ with $z(0)=z_{0}$. Clearly, $v(\cdot)$ and $\eta(\cdot)$ represent classical solutions to the corresponding problems associated with $A$ and $-A^{*}$. In particular, $z(t) \in \mathcal{D}(A) \times \mathcal{D}(A^{*})$ for any $t \geq 0$. This gives the validity of \eqref{EQ: SymplecticFormPreservationStationary} for mild (i.e. all) solutions starting from $\mathbb{L}^{+}_{b}$ by continuity. Now take any $z^{1}_{0},z^{2}_{0} \in \mathbb{L}^{+}_{b}$ and the corresponding solutions $z_{1}(\cdot), z_{2}(\cdot) \in L_{2}(0,\infty;\breve{\mathbb{H}})$ with $z_{1}(0)=z^{1}_{0}$ and $z_{2}(0) = z^{2}_{0}$. Then for any $t \geq 0$ we have
	\begin{equation}
		\label{EQ: StatDichPertbIsotropicSympIden}
		\langle z_{1}(t),Jz_{2}(t) \rangle_{\breve{\mathbb{H}}} = \langle z^{1}_{0},Jz^{2}_{0} \rangle_{\breve{\mathbb{H}}}.
	\end{equation}
	Passing to the limit as $t \to +\infty$ and using \eqref{EQ: ExpDecayStatHamiltonian}, we get $\langle z^{1}_{0},Jz^{2}_{0} \rangle_{\breve{\mathbb{H}}} = 0$. Thus, $\mathbb{L}^{+}_{b}$ is isotropic.
	
	By \eqref{EQ: StatDichSubspaceAsGraph}, $\mathbb{L}^{+}_{b}$ is a graph over the direct sum decomposition $\breve{\mathbb{H}} = \breve{\mathbb{H}}^{s} \oplus \breve{\mathbb{H}}^{u}$ into Lagrange subspaces (see Proposition \ref{PROP: StatDichSubspacesAreLagrange}). Now Theorem \ref{TH: LagrangeSubspaceCriterion} gives that $\mathbb{L}^{+}_{b}$ is Lagrange. The proof is finished.
\end{proof}

The presented proof reveals significant advantages of applying Lyapunov-Perron operators to linear Hamiltonian systems. Namely, in the nonlinear case, resolution (with unique solutions and uniform bounds) of fixed-point equations like \eqref{EQ: LyapPerronOperatorStationaryTrue} is possible only by applying the Contraction Mapping Theorem. In our case, such equations can be studied by more delicate methods of linear theory. For example, it is clear (see \eqref{EQ: ProofStatDichUtilizeFreqIneq}) that under \eqref{EQ: FreqConditionStationaryGeneral} the operator $\mathcal{T}$ is not necessarily contracting. Let us, however, mention an interesting case, where $\mathcal{T}$ can be proven to be a contraction.

Suppose $C \colon \mathbb{H} \to \mathbb{M}$, where $\mathbb{M}$ is a real Hilbert space, $\Lambda > 0$ and put $F_{1} := -\Lambda^{2}C^{*} C$, $F_{2} := 0$ and $F_{3} := \operatorname{Id}_{\mathbb{U}}$. Then the frequency inequality \eqref{EQ: FreqConditionStationaryGeneral} in terms of the transfer operator $W(p):=C(A-pI)^{-1}B$ can be described as
\begin{equation}
	\label{EQ: StatCaseSmithFreqCond}
	\sup_{\omega \in \mathbb{R}}\| W(i\omega) \|_{ \mathbb{U}^{\mathbb{C}} \to \mathbb{M}^{\mathbb{C}} } < \Lambda^{-1}.
\end{equation}
It can be verified that under this inequality, the operator $\mathcal{T}$ is a contraction. 

So, from the above, it is not surprising that direct applications of Lyapunov-Perron operators to nonlinear problems can only lead to conditions like \eqref{EQ: StatCaseSmithFreqCond}. For example, such conditions were explored by R.A.~Smith \cite{Smith1992} and M.~Miklav\v{c}i\v{c} \cite{Miclavcic1991}. In the case $A$ is self-adjoint, from \eqref{EQ: StatCaseSmithFreqCond} one can deduce the well-known Spectral Gap Condition in its optimal form (see \cite{Anikushin2020FreqParab}).

We refer to our work \cite{Anikushin2020FreqParab} for discussions concerned with optimality of the frequency inequality \eqref{EQ: FreqConditionStationaryGeneral}. It is however pretty clear from the proof of Theorem \ref{TH: DichotomyStatOptimizationLP} and especially \eqref{EQ: ProofStatDichUtilizeFreqIneq} that violating the inequality would likely prevent $\mathbb{L}^{+}_{b}$ to exist (at least) as a (in a sense uniform) perturbation of $\breve{\mathbb{H}}^{s}$. Thus, in terms of the perturbation problem it gives the maximum that can be achieved.

It should be also noted that there are many works containing ineffective and/or nonoptimal conditions even in the self-adjoint case (see, for example, Chapters VIII and IX in the monograph of R.~Temam \cite{Temam1997} or discussions in \cite{Anikushin2020FreqDelay,Anikushin2020FreqParab}). In the derivation of optimal conditions, the key element is the Fourier transform. In the field of inertial manifolds this was firstly done in the mentioned works of M.~Miklav\v{c}i\v{c} \cite{Miclavcic1991} and, in an implicit manner, R.A.~Smith \cite{Smith1992}. However, it seems that V.M.~Popov (see \cite{Popov1961} and references therein) was the first to apply the Fourier transform in the study of Lyapunov-Perron operators\footnote{Being a mathematical engineer, his initial motivation was to provide conditions for the global stability of certain control systems in terms of the transfer operator. So, there was no way for him to avoid the transition into the frequency domain (Fourier space).}. In the community of mathematical engineers, his method is known as the ``method of a priori integral estimates''.

\subsection{Nonoscillation}
Now we proceed to further investigations of the stationary case. For applications, the existence of a Lagrange subspace $\mathbb{L}^{+}_{b}$ is only half the battle. It is also important to know whether there exists a self-adjoint operator $P \in \mathcal{L}(\mathbb{H})$ such that
\begin{equation}
	\label{EQ: StatCaseNonoscillation}
	\mathbb{L}^{+}_{b} = \mathbb{L}_{P} := \{ (v, -Pv) \ | \ v \in \mathbb{H}  \}.
\end{equation}
In this case one says that \eqref{EQ: StatHamiltonianSystem} is \textit{nonoscillating}. Note that \eqref{EQ: StatCaseNonoscillation} means $\mathbb{L}^{+}_{b}$ is transversal to $\mathbb{L}_{\eta}$, where $\mathbb{L}_{\eta} := \{0_{\mathbb{H}}\} \times \mathbb{H}$, or, in other words, that it belongs to the local chart of $\mathbb{L}_{v} = \mathbb{L}^{\bot}_{\eta}$ in the Lagrangian-Grassmanian (see Proposition 2.21 in \cite{Furutani2004FLG}). Thus, \eqref{EQ: StatCaseNonoscillation} is an open condition.

Note that although the answer for \eqref{EQ: StatCaseNonoscillation} to hold in the stationary case is well-known (see Proposition \ref{PROP: StatControllabilitySufficientNonOsc}) and it is given by the $L_{2}$-controllability of the pair $(A,B)$ (in addition to the frequency inequality \eqref{EQ: FreqConditionStationaryGeneral}), here we are aimed to illustrate the problem in terms that are also suitable for studying nonstationary problems (see Section \ref{SEC: SpatAvgNonstatOpt}).

An immediate corollary from the construction of $\mathbb{L}^{+}_{b}$ is the following.
\begin{corollary}
	\label{COR: StatDichLagrangeSubsFredholmPair}
	Under the conditions of Theorem \ref{TH: DichotomyStatOptimizationLP}, for the Lagrange subspaces $\mathbb{L}^{+}_{b}$ and $\mathbb{L}_{\eta} := \{0_{\mathbb{H}}\} \times \mathbb{H}$ we have
	\begin{enumerate}
		\item[1)] $\dim (\mathbb{L}^{+}_{b} \cap \mathbb{L}_{\eta}) \leq j$;
		\item[2)] $\mathbb{L}^{+}_{b} + \mathbb{L}_{\eta}$ is closed and has finite codimension $\leq j$,
	\end{enumerate}
	where $j \geq 0$ from \nameref{DESC: DICHSTATIONARY}. In particular, $(\mathbb{L}^{+}_{b}, \mathbb{L}_{\eta})$ is a Fredholm pair (see \cite{Furutani2004FLG}).
\end{corollary}
\begin{proof}
	For item 1), from \eqref{EQ: StatDichSubspaceAsGraph} we get that $\mathbb{L}^{+}_{b} \cap \breve{\mathbb{H}}^{u} = \{ 0 \}$. Now we apply Lemma \ref{LEM: IntersectionDimensionLemma} to $\mathbb{H}_{f} := \{0_{\mathbb{H}}\} \times \mathbb{H}^{s}_{A^{*}}$, $\mathbb{H}_{a} := \{0_{\mathbb{H}}\} \times \mathbb{H}^{u}_{A^{*}}$ and $\mathbb{L} := \mathbb{L}^{+}_{b} \cap (\{0_{\mathbb{H}}\} \times \mathbb{H})$. 
	
	For item 2) we have
	\begin{equation}
		\mathbb{L}^{+}_{b} + \mathbb{L}_{\eta} = \{ (v,\eta) \in \mathbb{H} \times \mathbb{H} \ | \ (v,0) \in \Pi_{v}\mathbb{L}^{+}_{b}, \eta \in \mathbb{H} \},
	\end{equation}
	where $\Pi_{v}$ is the orthogonal projector onto $\mathbb{L}_{v} := \mathbb{H} \times \{ 0_{\mathbb{H}} \}$. From \eqref{EQ: StatDualCouplingStableUnstableSubspaces} and \eqref{EQ: StatDichSubspaceAsGraph} we deduce that $\Pi_{v} \breve{\mathbb{H}}^{s} =  \mathbb{H}^{s}_{A} \times \{0_{\mathbb{H}}\}$ and $\Pi_{v}M^{+} \breve{\mathbb{H}}^{s} \subset \mathbb{H}^{u}_{A} \times \{0_{\mathbb{H}}\}$ with $\dim \mathbb{H}^{u}_{A} = j$. From this it can be deduced that $\Pi_{v} \mathbb{L}^{+}_{b}$ is closed and there exists a finite-dimensional subspace $\mathbb{L}_{c} \subset \mathbb{H}^{u}_{A} \times \{0_{\mathbb{H}}\}$ such that $\Pi_{v} \mathbb{L}^{+}_{b} \oplus \mathbb{L}_{c} = \mathbb{H} \times \mathbb{H}$. This shows item 2). The proof is finished.
\end{proof}

The following proposition provides a useful characterization of the nonoscillation condition \eqref{EQ: StatCaseNonoscillation}.
\begin{proposition}
	\label{PROP: NonoscillationStationaryAsTrivialIntersection}
	Under the conditions of Theorem \ref{TH: DichotomyStatOptimizationLP}, \eqref{EQ: StatCaseNonoscillation} holds if and only if $\mathbb{L}^{+}_{b}$ intersects $\mathbb{L}_{\eta} = \{0_{\mathbb{H}}\} \times \mathbb{H}$ trivially.
\end{proposition}
\begin{proof}
	Note that the ``only if'' part is obvious. For the converse implication, we apply Theorem \ref{TH: CriterionLagrangeSubspaceIsAGraph} to $\mathbb{L}_{0} := \mathbb{L}^{+}_{b}$ and $\mathbb{L} := \mathbb{L}_{v} = \mathbb{H} \times \{0_{\mathbb{H}}\}$. Here item 3 of the theorem is satisfied since $\mathbb{L}_{b} \cap \mathbb{L}_{\eta} = \{ 0\}$ and the direct sum $\mathbb{L}_{b} \oplus \mathbb{L}_{\eta}$ can be complemented by a finite-dimensional (hence closed) subspace due to Corollary \ref{COR: StatDichLagrangeSubsFredholmPair}. Thus, the orthogonal projector $\Pi_{v}$ maps $\mathbb{L}^{+}_{b}$ bijectively onto $\mathbb{L}_{v}$ and, consequently, the inverse mapping is of the form $\mathbb{L}_{v} \ni (v,0) \mapsto (v,-Pv) \in \mathbb{L}^{+}_{b}$ for $v \in \mathbb{H}$. Since $\mathbb{L}^{+}_{b}$ is Lagrange, $P$ must be self-adjoint. The proof is finished.
\end{proof}

Recall that the pair $(A,B)$ is called $L_{2}$-controllable if for any $v_{0} \in \mathbb{H}$ there exists $\xi(\cdot) \in L_{2}(0,\infty;\mathbb{U})$ such that the mild solution $v(\cdot)$ to
\begin{equation}
	\label{EQ: StatControlSystemNonoscExpl}
	\dot{v}(t) = Av(t) + B\xi(t)
\end{equation}
with $v(0) = v_{0}$ belongs to $L_{2}(0,\infty;\mathbb{H})$.

It is not hard to see that this condition is necessary for the nonoscillation as follows.
\begin{lemma}
	Under the hypotheses of Theorem \ref{TH: DichotomyStatOptimizationLP}, suppose that $\mathbb{L}^{+}_{b}$ satisfies \eqref{EQ: StatCaseNonoscillation}. Then the pair $(A,B)$ is $L_{2}$-controllable.
\end{lemma}
\begin{proof}
	Indeed, take $v_{0} \in \mathbb{H}$. Then, by Theorem \ref{TH: DichotomyStatOptimizationLP}, there exists a solution $(v(\cdot),\eta(\cdot)) \in L_{2}(0,\infty;\breve{\mathbb{H}})$ to \eqref{EQ: StatHamiltonianSystem} such that $v(0) = v_{0}$ and $\eta(0) = -Pv_{0}$. Put $\xi(\cdot) := -F^{-1}_{3} F_{2} v(\cdot) + F^{-1}_{3} B^{*} \eta(\cdot)$. Clearly, $\xi(\cdot) \in L_{2}(0,\infty;\mathbb{U})$ and \eqref{EQ: StatControlSystemNonoscExpl} holds. The proof is finished.
\end{proof}

Let $\mathfrak{M}_{v_{0}}$ be the set of processes through $v_{0} \in \mathbb{H}$, i.e. the space of pairs $(v(\cdot),\xi(\cdot)) \in L_{2}(0,\infty;\mathbb{H}) \times L_{2}(0,\infty;\mathbb{U})$ which satisfy \eqref{EQ: StatControlSystemNonoscExpl} and $v(0) = v_{0}$. Under the $L_{2}$-controllability, the set $\mathfrak{M}_{v_{0}}$ is a nonempty closed affine subspace in $L_{2}(0,\infty;\mathbb{H}) \times L_{2}(0,\infty;\mathbb{U})$ given by a proper translation of the closed subspace $\mathfrak{M}_{0}$. Consider the space of processes $\mathcal{Z}$ given by the union of all $\mathfrak{M}_{v_{0}}$ and equipped with the norm
\begin{equation}
	\label{EQ: StatHamiltSpaceOfProcessesNorm}
	\|(v(\cdot),\xi(\cdot))\|^{2}_{\mathcal{Z}} := |v_{0}|^{2}_{\mathbb{H}} + \|v(\cdot)\|^{2}_{L_{2}(0,\infty;\mathbb{H})} + \|\xi(\cdot)\|^{2}_{L_{2}(0,\infty;\mathbb{U})}
\end{equation}
making it a Banach space (see \cite{Anikushin2020FreqDelay}).

On $\mathcal{Z}$ there is a well-defined integral quadratic functional
\begin{equation}
	\label{EQ: IntegralQuadraticFunctionalStationary}
	\mathcal{J}_{\mathcal{F}}(v(\cdot),\xi(\cdot)) := \int_{0}^{+\infty}\mathcal{F}(v(t),\xi(t))dt.
\end{equation}

A key property of solutions lying on $\mathbb{L}^{+}_{b}$ is related to critical points of $\mathcal{J}_{\mathcal{F}}$ as in the following lemma.
\begin{lemma}
	\label{LEM: StationaryStabLagrangeCritPoints}
	Under the hypotheses of Theorem \ref{TH: DichotomyStatOptimizationLP}, suppose that the pair $(A,B)$ is $L_{2}$-controllable. Take any $(v_{0},\eta_{0}) \in \mathbb{L}^{+}_{b}$ and consider the corresponding solution $(v(\cdot),\eta(\cdot))$ of \eqref{EQ: StatHamiltonianSystem} with $v(0) = v_{0}$ and $\eta(0) = \eta_{0}$. Put $\xi(\cdot) := -F^{-1}_{3}F_{2}v(\cdot) + F^{-1}_{3}B^{*}\eta(\cdot)$. Then $(v(\cdot),\xi(\cdot)) \in \mathfrak{M}_{v_{0}}$ and it is a critical point for $\mathcal{J}_{\mathcal{F}}$ on $\mathfrak{M}_{v_{0}}$.
\end{lemma}
\begin{proof}
	A critical point satisfies 
	\begin{equation}
		\label{EQ: CriticalPointIdentityStationary}
		\int_{0}^{\infty}\langle F_{1}v(t) + F^{*}_{2}\xi(t),v_{h}(t)\rangle_{\mathbb{H}} dt + \int_{0}^{\infty}\langle F_{2}v(t) + F_{3} \xi(t), \xi_{h}(t)\rangle_{\mathbb{U}}dt = 0.
	\end{equation}
	for any $(v_{h}(\cdot),\xi_{h}(\cdot)) \in \mathfrak{M}_{0}$. Using $\dot{\eta} + A^{*}\eta(t) = F_{1}v(t) + F^{*}_{2}\xi(t)$ and integrating by parts, we obtain\footnote{The identity can be justified using approximation of $F_{1}v(t) + F^{*}_{2}\xi(t)$ by smooth functions and $\eta(T)$ by elements of $\mathcal{D}(A^{*})$.} for any $T>0$
	\begin{equation}
		\int_{0}^{T}\langle F_{1}v(t) + F^{*}_{2}\xi(t),v_{h}(t)\rangle_{\mathbb{H}} dt = \langle \eta(T), v_{h}(T) \rangle_{\mathbb{H}} - \int_{0}^{T}\langle \eta(t), B\xi_{h}(t)\rangle_{\mathbb{H}} dt.
	\end{equation}
	Taking it to the limit as $T \to \infty$ and substituting it into \eqref{EQ: CriticalPointIdentityStationary}, we get the conclusion. The proof is finished.
\end{proof}

Now let us establish the following proposition relating nonoscillation and uniqueness of critical points of $\mathcal{J}_{\mathcal{F}}$. 
\begin{proposition}
	\label{PROP: NonoscillationUniqueFixedPoint}
	Under the hypotheses of Theorem \ref{TH: DichotomyStatOptimizationLP}, suppose that the pair $(A,B)$ is $L_{2}$-controllable and that the critical point of $\mathcal{J}_{\mathcal{F}}$ on $\mathfrak{M}_{0}$ is unique. Then the nonoscillation condition \eqref{EQ: StatCaseNonoscillation} holds.
\end{proposition}
\begin{proof}
	Supposing the contrary, there must exist a point $(0,\eta_{0}) \in \mathbb{L}^{+}_{b}$ with $\eta_{0} \not=0$ due to Proposition \ref{PROP: NonoscillationStationaryAsTrivialIntersection}. Let $(v(\cdot),\eta(\cdot))$ be the corresponding solution of \eqref{EQ: StatHamiltonianSystem} with $v(0) = 0$ and $\eta(0) = \eta_{0}$. By Lemma \ref{LEM: StationaryStabLagrangeCritPoints}, $(v(\cdot),\xi(\cdot)) \in \mathfrak{M}_{0}$, where $\xi(\cdot) = -F^{-1}_{3}F_{2}v(\cdot) + F^{-1}_{3}B^{*}\eta(\cdot)$, is a critical point and, consequently, $v(\cdot) \equiv 0$ and $\xi(\cdot) \equiv 0$. This implies that $\eta(\cdot)$ must satisfy $\dot{\eta} = -A^{*}\eta(t)$ and $B^{*}\eta(t) = 0$ for $t \geq 0$  and, consequently, $\eta_{0} \in \mathbb{H}^{u}_{A^{*}}$. Let $\mathbb{L}_{A^{*}}$ be the space of all such $\eta_{0}$. Then $\mathbb{L}_{A^{*}} \subset \mathbb{H}^{u}_{A^{*}}$ is a spectral subspace for $A^{*}$ on which $B^{*}$ vanish. Let $\Pi_{\mathbb{L}_{A^{*}}}$ be the spectral projector onto that subspace. For any $v,\eta \in \mathbb{H}$ we have
	\begin{equation}
		0 = \langle B^{*}\Pi_{\mathbb{L}_{A^{*}}} \eta, v \rangle  = \langle \eta, \Pi^{*}_{\mathbb{L}_{A^{*}}}Bv \rangle.
	\end{equation}
	Note that $\Pi^{*}_{\mathbb{L}_{A^{*}}} = \Pi_{\mathbb{L}_{A}}$ is a spectral projector for $A$ onto a subspace $\mathbb{L}_{A} \subset \mathbb{H}^{u}_{A}$ and the above identity gives that $\Pi_{\mathbb{L}_{A}} B = 0$. In particular, $(\Pi_{\mathbb{L}_{A}}A, \Pi_{\mathbb{L}_{A}}B)$ is not $L_{2}$-controllable in $\mathbb{L}_{A}$ that leads to a contradiction. The proof is finished.
\end{proof}

In fact, the converse to the above proposition also holds and, moreover, in this case $\mathcal{J}_{\mathcal{F}}$ must admit a unique minimum on each $\mathfrak{M}_{v_{0}}$. This will be established in a series of statements.

\begin{proposition}
	\label{LEM: PSatRiccatiEquationStationary}
	Under the hypotheses of Theorem \ref{TH: DichotomyStatOptimizationLP}, suppose that $\mathbb{L}^{+}_{b}$ satisfies \eqref{EQ: StatCaseNonoscillation} with some $P=P^{*} \in \mathcal{L}(\mathbb{H})$. Then $P$ takes $\mathcal{D}(A)$ into $\mathcal{D}(A^{*})$ and satisfies the Riccati equation on $\mathcal{D}(A)$ as
	\begin{equation}
		\label{EQ: RiccattiEquationStatHamilt}
		\begin{split}
			-PH_{3}P + PH_{1} + H^{*}_{1} P + H_{2} = 0 \textit{ or }\\
			-P BF^{-1}_{3} B^{*} P + P (A - BF^{-1}_{3}F_{2})+ \\+ (A^{*} - F^{*}_{2}F^{-1}_{3}B^{*} )P + F_{1} - F^{*}_{2} F^{-1}_{3} F_{2} = 0
		\end{split}
	\end{equation}
	where $H_{1}$, $H_{2}$ and $H_{3}$ are given by \eqref{EQ: HamiltonianStationary}.
\end{proposition}
\begin{proof}
	Take $v_{0} \in \mathcal{D}(A)$, $\eta_{0} = -Pv_{0}$ and consider the corresponding solutions $v(\cdot)$ and $\eta(\cdot)$ with $v(0)=v_{0}$ and $\eta(0) = \eta_{0}$. Clearly, $v(\cdot)$ satisfies
	\begin{equation}
		\dot{v}(t) = (A + K)v(t)
	\end{equation}
	 with $K = -F^{-1}_{3}F_{2} - F^{-1}_{3}B^{*} P$. Since $K$ is bounded, $A+K$ generates a $C_{0}$-semigroup with the same domain $\mathcal{D}(A)$ (see \cite{EngelNagel2000} or \cite{Krein1971}). Thus, $v(\cdot)$ is a classical solution, i.e. $v(\cdot) \in C^{1}([0,\infty);\mathbb{H}) \cap C([0,\infty);\mathcal{D}(A))$. From this we have that $\eta(\cdot)$ is a mild solution to
	 \begin{equation}
	 	\dot{\eta}(t) = -A^{*}\eta(t) + (F_{1} + F^{*}_{2}K) v(t)
	 \end{equation}
	 Since $\eta(\cdot)=Pv(\cdot)$ and $(F_{1} + F^{*}_{2}K) v(\cdot)$ belong to $C^{1}([0,\infty);\mathbb{H})$, by the Cauchy formula, we must have $\eta(t) \in \mathcal{D}(A^{*})$ for any $t \geq 0$. In particular, $\eta(0) = Pv_{0} \in \mathcal{D}(A^{*})$.
	 
	 Since $\mathbb{L}^{+}_{b}$ is invariant, for any $v_{0} \in \mathcal{D}(A)$ we have
	 \begin{equation}
	 	H\begin{pmatrix}
	 		v_{0}\\
	 		-Pv_{0}
	 	\end{pmatrix} = 
	 	\begin{pmatrix}
	 		H_{1}v_{0} - H_{3}Pv_{0}\\
	 		H_{2}v_{0} + H^{*}_{1}Pv_{0}
	 	\end{pmatrix} \in \mathbb{L}^{+}_{b}
	 \end{equation}
	 and since $\mathbb{L}^{+}_{b}$ is isotropic, for any $v \in \mathbb{H}$ we get
	 \begin{equation}
	 	0 = \left\langle H\begin{pmatrix}
	 		v_{0}\\
	 		-Pv_{0}
	 	\end{pmatrix}, \begin{pmatrix}
	 	Pv\\
	 	v
	 	\end{pmatrix} \right\rangle_{\breve{\mathbb{H}}}
	 \end{equation}
	 that is equivalent to \eqref{EQ: RiccattiEquationStatHamilt} by straightforward computations. The proof is finished.
\end{proof}

It is well-known in the calculus of variations that solvability of Riccati equations is related to sign-definiteness of the second variation.

\begin{proposition}
	\label{PROP: RiccatiStationaryIntegralIdentity}
	Let $P=P^{*} \in \mathcal{L}(\mathbb{H})$ be a solution to the Riccati equation \eqref{EQ: RiccattiEquationStatHamilt}. Then for any $v(\cdot) \in L_{2}(0,T;\mathbb{H})$ and $\xi(\cdot) \in L_{2}(0,T;\mathbb{U})$ satisfying \eqref{EQ: StatControlSystemNonoscExpl} on $[0,T]$ for some $T>0$ we have
	\begin{equation}
		\label{EQ: RiccatiStationaryIntegralFormula}
		\begin{split}
			V_{P}(v(T)) - V_{P}(v(0)) + \int_{0}^{T}\mathcal{F}(v(t),\xi(t))dt =\\= \int_{0}^{T}\langle F_{3}(\xi(t) - Kv(t)), \xi(t) - K v(t)\rangle_{\mathbb{U}}dt,
		\end{split}
	\end{equation}
	 where $V_{P}(v):=\langle v, Pv \rangle_{\mathbb{H}}$ is the quadratic form of $P$ and $K:= -F^{-1}_{3}F_{2} - F^{-1}_{3}B^{*}P$.
\end{proposition}
\begin{proof}
	It is sufficient to consider solutions with $v(0) = v_{0} \in \mathcal{D}(A)$ and $\xi(\cdot) \in C^{1}([0,T];\mathbb{U})$ and then apply the continuity argument. In particular, $v(\cdot)$ is classical and $v(\cdot) \in C^{1}([0,T];\mathbb{H}) \cap C([0,T];\mathcal{D}(A))$. Differentiating \eqref{EQ: RiccatiStationaryIntegralFormula} by $T$, we get that it is sufficient to show that
	\begin{equation}
		\label{EQ: RiccatiStationaryDifferentialFormula}
		2\langle Av + B\xi, Pv\rangle_{\mathbb{H}} + \mathcal{F}(v,\xi) = \langle F_{3}(\xi - Kv), \xi - K v\rangle_{\mathbb{U}}
	\end{equation}
	for any $v \in \mathcal{D}(A)$ and $\xi \in \mathbb{U}$. This can be done by expanding the right-hand side in \eqref{EQ: RiccatiStationaryDifferentialFormula} according to \eqref{EQ: RiccattiEquationStatHamilt}. 
\end{proof}

In the context of Proposition \ref{PROP: RiccatiStationaryIntegralIdentity}, the operator $K$ does not necessarily stabilize the system. So, one cannot deduce the existence of optimal processes just from an abstract solution $P$ of the Riccati equation. However, this is the case under \eqref{EQ: StatCaseNonoscillation}.
\begin{theorem}
	Under the hypotheses of Theorem \ref{TH: DichotomyStatOptimizationLP}, suppose that $\mathbb{L}^{+}_{b}$ satisfies the nonoscillation condition \eqref{EQ: StatCaseNonoscillation} with some $P=P^{*} \in \mathcal{L}(\mathbb{H})$. Then for each $v_{0} \in \mathbb{H}$ the functional $\mathcal{J}_{\mathcal{F}}$ from \eqref{EQ: IntegralQuadraticFunctionalStationary} admits a unique minimum on $\mathfrak{M}_{v_{0}}$ given by $\langle v_{0}, -Pv_{0} \rangle_{\mathbb{H}}$ and the corresponding optimal process $(v^{0}(\cdot;v_{0}), \xi^{0}(\cdot;v_{0})) \in \mathfrak{M}_{v_{0}}$ satisfies $\xi^{0}(\cdot;v_{0}) = K v^{0}(\cdot;v_{0})$, where $K:= -F^{-1}_{3}F_{2} - F^{-1}_{3}B^{*}P$. 
	
	Moreover, there exist $\varepsilon>0$ and $P_{\varepsilon}=P^{*}_{\varepsilon} \in \mathcal{L}(\mathbb{H})$ such that for any $v(\cdot) \in L_{2}(0,T;\mathbb{H})$ and $\xi(\cdot) \in L_{2}(0,T;\mathbb{U})$ satisfying \eqref{EQ: StatControlSystemNonoscExpl} on $[0,T]$ (with some $T>0$) we have
	\begin{equation}
		\label{EQ: LyapInequalityStatProblem}
		V_{P_{\varepsilon}}(v(T)) - V_{P_{\varepsilon}}(v(0)) + \int_{0}^{T}\mathcal{F}(v(t),\xi(t))dt \geq \varepsilon \int_{0}^{T}(|v(t)|^{2}_{\mathbb{H}}+|\xi(t)|^{2}_{\mathbb{U}})dt.
	\end{equation}
\end{theorem}
\begin{proof}
	For the first part, we just use Lemma \ref{LEM: PSatRiccatiEquationStationary} and Proposition \ref{PROP: RiccatiStationaryIntegralIdentity}.
	
	For the second part, note that the above results can be applied to the quadratic form $\mathcal{F}_{\varepsilon}(v,\xi) := \mathcal{F}(v,\xi) - \varepsilon (|v|^{2}_{\mathbb{H}} + |\xi|^{2}_{\mathbb{U}})$ for a sufficiently small $\varepsilon>0$. Indeed, the frequency inequality \eqref{EQ: FreqConditionStationaryGeneral} would be preserved and the new space $\mathbb{L}^{+}_{b}$ constructed by Theorem \ref{TH: DichotomyStatOptimizationLP} would be slightly perturbed as follows from the method of construction. Since the nonoscillation \eqref{EQ: StatCaseNonoscillation} is an open condition, it is also preserved with some operator $P_{\varepsilon}$ which tends to $P$ as $\varepsilon \to 0+$ in the operator norm. Then Lemma \ref{LEM: PSatRiccatiEquationStationary} and Proposition \ref{PROP: RiccatiStationaryIntegralIdentity} applied to $P_{\varepsilon}$ and $\mathcal{F}_{\varepsilon}$ along with the positive-definiteness of $F_{3}$ give \eqref{EQ: LyapInequalityStatProblem}. The proof is finished.
\end{proof}

Finally, let us show that the $L_{2}$-controllability is sufficient for the nonoscillation.
\begin{proposition}
	\label{PROP: StatControllabilitySufficientNonOsc}
	Under the hypotheses of Theorem \ref{TH: DichotomyStatOptimizationLP}, suppose that the pair $(A,B)$ is $L_{2}$-controllable. Then for any $(v(\cdot),\xi(\cdot)) \in \mathfrak{M}_{0}$ we have
	\begin{equation}
		\label{EQ: StatHamiltonianFreqCondCoerciveFunc}
		\mathcal{J}_{\mathcal{F}}(v(\cdot),\xi(\cdot)) \geq \frac{\delta}{\delta M^{2} + 1} \left( \|v(\cdot)\|^{2}_{L_{2}(0,\infty;\mathbb{H})} + \|\xi(\cdot)\|^{2}_{L_{2}(0,\infty;\mathbb{U})} \right),
	\end{equation}
	where $\delta$ is given by \eqref{EQ: FreqConditionStationaryGeneral} and $M := \sup_{\omega \in \mathbb{R}} \| (A - i\omega I)^{-1}B\|$. In particular, $\mathbb{L}^{+}_{b}$ satisfies the nonoscillation condition \eqref{EQ: StatCaseNonoscillation} with some $P=P^{*} \in \mathcal{L}(\mathbb{H})$.
\end{proposition}
\begin{proof}
	Let $\widehat{v}(\cdot)$ and $\widehat{\xi}(\cdot)$ be the Fourier transforms of $v(\cdot)$ and $\xi(\cdot)$ respectively after extending the functions by zero to the negative semi-axis. Then the Parseval identity, Lemma \ref{LEM: FourierSolutionOnRLemma} and \eqref{EQ: FreqConditionStationaryGeneral} give 
	\begin{equation}
		\begin{split}
			\mathcal{J}_{\mathcal{F}}(v(\cdot),\xi(\cdot)) = \int_{-\infty}^{+\infty} \mathcal{F}(\widehat{v}(\omega), \widehat{\xi}(\omega))d\omega \geq \delta \int_{-\infty}^{+\infty}|\widehat{\xi}(\omega)|^{2}_{\mathbb{U}^{\mathbb{C}}}d\omega \geq \\ \geq
			\frac{\delta}{\delta M^{2} + 1}\int_{-\infty}^{+\infty}\left(|\widehat{v}(\omega)|^{2}_{\mathbb{H}^{\mathbb{C}}} + |\widehat{\xi}(\omega)|^{2}_{\mathbb{U}^{\mathbb{C}}}\right)d\omega
		\end{split}
	\end{equation}
	and, consequently, \eqref{EQ: StatHamiltonianFreqCondCoerciveFunc}. 
	
	For the second part, \eqref{EQ: StatHamiltonianFreqCondCoerciveFunc} and Proposition \ref{PROP: NonoscillationUniqueFixedPoint} gives the desired. The proof is finished.
\end{proof}

For applications to stability problems or to the existence of inertial manifolds, the inequality \eqref{EQ: LyapInequalityStatProblem} and its analogs for nonstationary cases are crucial. We recall that the setting of a bounded quadratic form $\mathcal{F}$ is considered here for simplicity and many problems such as semilinear or quasi-linear parabolic equations or equations with discrete delays demand considerations of unbounded quadratic forms $\mathcal{F}$. Such studies are possible due to certain regularity or, more generally, structural properties of the corresponding control systems. 

Note that the most general context for stationary infinite-horizon quadratic optimization is presented in our work \cite{Anikushin2020FreqDelay}. However, the formalism required to study the ``Hamiltonian side'' in such a general context still requires developments. Although it seems that semilinear parabolic equations \cite{Anikushin2020FreqParab} can be studied with no problems, some important nuances arise in the study of delay equations. One of arising features is due to the taking adjoints of certain elements of $\mathcal{F}$ (explicitly or implicitly as in $F_{1}$ above \eqref{EQ: StatCaseSmithFreqCond}), which in the context of \cite{Anikushin2020FreqDelay} can be defined only on a Banach space, as it is required to derive the Hamiltonian $H$ in \eqref{EQ: HamiltonianStationary}. Another feature is that the action of $\mathcal{F}$ can be understood only in the integral sense due to structural properties of solutions which we call structural Cauchy formulas \cite{Anikushin2020FreqDelay,Anikushin2023Comp}. 

Such studies on stationary optimization have already led to constructing inertial manifolds (see our papers \cite{Anikushin2020FreqDelay,Anikushin2022Semigroups,Anikushin2020Geom,AnikushinAADyn2021}; and \cite{AnikushinRom2023SS} joint with A.O.~Romanov) and obtaining effective global stability criteria (see our work \cite{Anikushin2023Comp}; \cite{AnikushinRomanov2023FreqConds} joint with A.O.~Romanov) for nonlinear delay equations. We hope that delay equations can be also a motivating example in exploring the Hamiltonian side of the problem, especially for nonstationary optimization.
\section{Nonoscillation under the Spatial Averaging}

\label{SEC: SpatAvgNonstatOpt}

In this section, we are going to discuss relations between nonstationary quadratic optimization and the Spatial Averaging Principle. This method was introduced by J.~Mallet-Paret and G.R.~Sell \cite{MalletParetSell1988SA} to construct inertial manifolds for scalar 2D and 3D reaction-diffusion equations, where the Spectral Gap Condition does not hold. In \cite{MalletParetSell1988SA}, the key observation (which justifies the name) is that the Schr\"{o}dinger operator $\Delta + v(x)$ (arising from the linearization of scalar reaction-diffusion equations) can sometimes be nicely approximated on large spectral subspaces of the Laplacian by the operator $\Delta + \overline{v}$, where $\overline{v}$ is the mean value of $v$. This is caused by delicate number-theoretic properties of eigenvalues. 

Recently, the method was developed by S.~Zelik and A.~Kostianko to extend possible applications. This allowed them to study some types of Cahn-Hilliard equations, modifications of Navier-Stokes equations and the complex Ginzburg-Landau equation (see the survey of A.~Kostianko et al. \cite{KostiankoZelikSA2020}).

In \cite{Anikushin2020FreqParab} we posed the problem of establishing connections between the Spatial Averaging Principle and nonstationary quadratic optimization. Below we will show  the existence of nonoscillating stable Lagrange bundles for nonautonomous Hamiltonian systems related to the Spatial Averaging Principle, if conditions of the latter are satisfied.

Suppose $A_{0}$ is a positive-definite self-adjoint operator in a real Hilbert space $\mathbb{H}$ having compact resolvent. Let $0 < \lambda_{1} \leq \lambda_{2} \leq \ldots$ be the eigenvalues of $A_{0}$ and let $e_{1}, e_{2}, \ldots$ be the associated orthonormal basis of eigenvectors, i.e. $A_{0} e_{j} = \lambda_{j} e_{j}$ for any $j = 1,2,\ldots$.

Let $\langle \cdot, \cdot \rangle_{\mathbb{H}}$ be the inner product in $\mathbb{H}$. For any positive integers $k$ and $N$ we introduce the orthogonal spectral projectors given by
\begin{equation}
	\label{EQ: SpatAvgProjectorsDef}
	\begin{split}
		\mathcal{P}_{k,N}v &:=\sum_{j \colon \lambda_{j} < \lambda_{N} - k} \langle v,e_{j}\rangle_{\mathbb{H}}	e_{j}, \\ \mathcal{Q}_{k,N}v &:= \sum_{j \colon \lambda_{j} > \lambda_{N} + k} \langle v,e_{j} \rangle_{\mathbb{H}}	e_{j}, \\ \mathcal{I}_{k,N} &:= \operatorname{Id}_{\mathbb{H}} -  \mathcal{P}_{k,N} - \mathcal{Q}_{k,N}
	\end{split}
\end{equation}
for any $v \in \mathbb{H}$. Here $\mathcal{P}_{k,N}$, $\mathcal{Q}_{k,N}$ and $\mathcal{I}_{k,N}$ are called the projectors on \textit{essentially lower}, \textit{essentially higher} and \textit{intermediate} modes respectively.

Let us take $\mathbb{U} := \mathbb{H} \times \mathbb{H}$ as the control space. For any $\xi \in \mathbb{U}$ we write $\xi=(\xi_{I}, \xi_{c})$, where $\xi_{I}, \xi_{c} \in \mathbb{H}$. 

Let $(\mathcal{Q},\vartheta)$ be a semiflow on a complete metric space $\mathcal{Q}$ and consider a continuous function $a(\cdot) \colon \mathcal{Q} \to \mathbb{R}$ such that
\begin{equation}
	\label{EQ: SpatAvgFunctionABound}
	a_{b}:=\sup_{q \in \mathcal{Q}}|a(q)| \leq \Lambda + \delta,
\end{equation}
where $\Lambda$ and $\delta$ are two positive constants (see the following remark).

\begin{remark}
	In the context of the Spatial Averaging Principle, we deal with the nonautonomous system over $(\mathcal{Q},\vartheta)$ given by
	\begin{equation}
		\label{EQ: ExampleSpatAveraging}
		\dot{v}(t) = -A_{0}v(t) + L(\vartheta^{t}(q))v(t),
	\end{equation}
	where $L(q) \in \mathcal{L}(\mathbb{H})$ is a family of linear operators with the norm not exceeding $\Lambda$. Then conditions of the method particularly require for some $k$ and $N$ that
	\begin{equation}
		\label{EQ: SpatAvgMainCondition}
		\| \mathcal{I}_{k,N}L(q)\mathcal{I}_{k,N} - a(q) \mathcal{I}_{k,N} \|_{\mathcal{L}(\mathbb{H})} \leq \delta \text{ for any } q \in \mathcal{Q},
	\end{equation}
	that is $L(q)$ is approximated on the intermediate modes by the scalar operator $a(\cdot)$. Clearly, we have \eqref{EQ: SpatAvgFunctionABound} satisfied.
	
	To understand the forthcoming constructions, it is also convenient to consider \eqref{EQ: ExampleSpatAveraging} as the control system 
	\begin{equation}
		\label{EQ: SpatAvgControlSysExample}
		\dot{v}(t) = -A_{0}v(t) + \xi_{I}(t) + \xi_{c}(t)
	\end{equation}
	closed by the feedback 
	\begin{equation}
	\label{EQ: SpatAvgClosedFeedbackExample}
	\xi_{I}(t) = L(\vartheta^{t}(q))\mathcal{I}_{k,N}v(t) \text{ and } \xi_{c}(t) = L(\vartheta^{t}(q))(\mathcal{P}_{k,N}+\mathcal{Q}_{k,N})v(t).
	\end{equation}
	\qed
\end{remark}

For each $q \in \mathcal{Q}$, consider the quadratic forms of $v \in \mathbb{H}$ and $\xi \in \mathbb{U}$ given by
\begin{equation}
	\label{EQ: SpatAvrgQuadraticForms}
	\begin{split}
		\mathcal{F}_{1,q}(v,\xi) &:= |\mathcal{I}_{k,N}\xi_{I} - a(q) \mathcal{I}_{k,N}v |^{2}_{\mathbb{H}} - \delta^{2} |\mathcal{I}_{k,N}v|^{2}_{\mathbb{H}},\\
		\mathcal{F}_{2}(v,\xi) &:= |(\mathcal{P}_{k,N}+\mathcal{Q}_{k,N})\xi_{I}|^{2}_{\mathbb{H}} - \Lambda^{2} |\mathcal{I}_{k,N}v|^{2}_{\mathbb{H}},\\
		\mathcal{F}_{3}(v,\xi) &:= |\xi_{c}|^{2}_{\mathbb{H}} - \Lambda^{2} |(\mathcal{P}_{k,N}+\mathcal{Q}_{k,N})v|^{2}_{\mathbb{H}}\\
	\end{split}
\end{equation}
and sum them together with some constant coefficients $\tau_{1},\tau_{2},\tau_{3} > 0$ (to be determined) as
\begin{equation}
	\label{EQ: QuadraticFormSpatialAveraging}
	\mathcal{F}_{q}(v,\xi) := \tau_{1} \mathcal{F}_{1,q}(v,\xi) + \tau_{2} \mathcal{F}_{2}(v,\xi) + \tau_{3} \mathcal{F}_{3}(v,\xi).
\end{equation}

\begin{remark}
	The introduced quadratic forms define quadratic constraints for \eqref{EQ: SpatAvgControlSysExample} by the inequalities $\mathcal{F}_{1,q}(v,\xi) \leq 0$, $\mathcal{F}_{2}(v,\xi) \leq 0$ and $\mathcal{F}_{3}(v,\xi) \leq 0$. It is not hard to see that the constraints are satisfied if the closed feedback \eqref{EQ: SpatAvgClosedFeedbackExample} is applied and \eqref{EQ: SpatAvgMainCondition} holds. Here \eqref{EQ: QuadraticFormSpatialAveraging} represents a transition (known as the S-procedure; see \cite{Gelig1978}) from the triple constraint to a single constraint. It allows to reduce (usually, with a loss of information in the case of several constraints) the nonconvex problem to a convex one, which we are aimed to study.
\end{remark}

Clearly, we have $\mathcal{F}_{q}(v,\xi) = (F_{1}(q)v,v)_{\mathbb{H}} + 2(F_{2}(q)v,\xi)_{\mathbb{H}} + (F_{3}(q)\xi,\xi)_{\mathbb{U}}$, where the operators $F_{1}=F^{*}_{1} \in \mathcal{L}(\mathbb{H})$, $F_{2} \in \mathcal{L}(\mathbb{H};\mathbb{U})$ and $F_{3} = F^{*}_{3} \in \mathcal{L}(\mathbb{U})$ are given by
\begin{equation}
	\label{EQ: SpatAvegQuadraticFormCoeff}
	\begin{split}
		F_{1}(q)v &:=  \tau_{1}|a(q)|^{2}\mathcal{I}_{k,N}v - \tau_{1}\delta^{2} \mathcal{I}_{k,N}v - \tau_{2}\Lambda^{2}\mathcal{I}_{k,N}v - \tau_{3}\Lambda^{2}(\mathcal{P}_{k,N} + \mathcal{Q}_{k,N})v,\\
		F_{2}(q)v &:= (-\tau_{1}a(q) \mathcal{I}_{k,N}v,0_{\mathbb{H}}),\\
		F_{3}(q)\xi = F_{3}\xi &:= (\tau_{1} \mathcal{I}_{k,N}\xi_{I} + \tau_{2}(\mathcal{P}_{k,N} + \mathcal{Q}_{k,N})\xi_{I}, \tau_{3} \xi_{c})
	\end{split}
\end{equation}
for $v \in \mathbb{H}$ and $\xi \in \mathbb{U}$.

We define the control operator $B \colon \mathbb{U} \to \mathbb{H}$ as $B\xi := \xi_{I} + \xi_{c}$. Moreover, for a given $N$, let us put $\nu_{0}(q) := \alpha - a(q)$ for $\alpha := (\lambda_{N}+\lambda_{N+1})/2$ and consider the operator $A(q) := -A_{0} + \nu_{0}(q)I$.

We are going to study the nonautonomous Hamiltonian system over $(\mathcal{Q},\vartheta)$ associated with $A(q)$, $B$ and $\mathcal{F}_{q}$ as
\begin{equation}
	\label{EQ: NonStatHamiltonianSystemSpatAv}
	\begin{pmatrix}
		\dot{v}(t)\\
		\dot{\eta}(t)
	\end{pmatrix} = H(\vartheta^{t}(q))
	\begin{pmatrix}
		v(t)\\
		\eta(t)	
	\end{pmatrix},
\end{equation}
where
\begin{equation}
	H(q) := \begin{pmatrix}
		\widehat{A}(q) & B F^{-1}_{3} B^{*}\\
		F_{1}(q) - F^{*}_{2}(q) F^{-1}_{3}F_{2}(q) & -\widehat{A}^{*}(q)
	\end{pmatrix},
\end{equation}
and $\widehat{A}(q) := A(q) - B F^{-1}_{3} F_{2}(q)$.

\subsection{Stable Lagrange bundles}
Let us put $A:= -A_{0} + \alpha I$ (do not confuse with $A(q)$) recalling $\alpha = (\lambda_{N}+\lambda_{N+1})/2$ and rewrite \eqref{EQ: NonStatHamiltonianSystemSpatAv} as
\begin{equation}
	\label{EQ: NonStatHamiltonianSystemSpatAvRewritten}
	\begin{pmatrix}
		\dot{v}(t)\\
		\dot{\eta}(t)
	\end{pmatrix} =
	\begin{pmatrix}
		A v(t)\\
		-A\eta(t)	
	\end{pmatrix} + R(\vartheta^{t}(q))\begin{pmatrix}
	v(t)\\
	\eta(t)	
	\end{pmatrix},
\end{equation}
where
\begin{equation}
	\label{EQ: NonStatHamiltonianSpatAvPerturbationMatrixR}
	R(q) = \begin{pmatrix}
		- B F^{-1}_{3} F_{2}(q) - a(q) I & B F^{-1}_{3} B^{*}\\
		F_{1}(q) - F^{*}_{2}(q) F^{-1}_{3}F_{2}(q) & (B F^{-1}_{3} F_{2}(q))^{*} + a(q) I
	\end{pmatrix}.
\end{equation}

From \eqref{EQ: SpatAvegQuadraticFormCoeff} and the above definition of $B$, it is not hard to see that 
\begin{equation}
	\begin{split}
		-BF^{-1}_{3}F_{2}(q)v &= a(q) \mathcal{I}_{k,N} v,\\ BF^{-1}_{3}B^{*} \eta &= \tau^{-1}_{1}\mathcal{I}_{k,N}\eta + \tau^{-1}_{2}(\mathcal{P}_{k,N} + \mathcal{Q}_{k,N})\eta + \tau^{-1}_{3}\eta,\\
		\left[F_{1}(q) - F^{*}_{2}(q) F^{-1}_{3} F_{2}(q) \right] v &=\\= (\tau_{1}-1)|a(q)|^{2} \mathcal{I}_{k,N}v &- \tau_{1}\delta^{2} \mathcal{I}_{k,N} v - \tau_{2}\Lambda^{2} \mathcal{I}_{k,N} v - \tau_{3} \Lambda^{2}(\mathcal{P}_{k,N} + \mathcal{Q}_{k,N})v, \\ (BF^{-1}_{3}F_{2}(q))^{*}\eta &= -a(q) \mathcal{I}_{k,N} \eta
	\end{split}
\end{equation}
for any $v,\eta \in \mathbb{H}$. For $\overline{\mu}:=(\lambda_{N+1}-\lambda_{N})/2 = -\lambda_{N} + \alpha$, let us put
\begin{equation}
	\label{EQ: SpatAvrgTauChoice}
	\tau_{1} := 1, \ \tau_{2} := \frac{1}{4}\left(\frac{\overline{\mu}}{\Lambda}\right)^{2} \text{ and } \tau_{3} := 1.
\end{equation}
Then \eqref{EQ: NonStatHamiltonianSystemSpatAvRewritten} reads as
\begin{equation}
	\label{EQ: SpatAvNonstatOptConcreteEquation}
	\begin{split}
		\dot{v}(t) &= A v(t) - a(\vartheta^{t}(q))(\mathcal{P}_{k,N} + \mathcal{Q}_{k,N})v(t) + \mathcal{I}_{k,N}\eta(t) + \eta(t) +\\&+ 4\left(\frac{\Lambda}{\overline{\mu}}\right)^{2}(\mathcal{P}_{k,N} + \mathcal{Q}_{k,N})\eta(t),\\
		\dot{\eta}(t) &= -A \eta(t) + a(\vartheta^{t}(q))(\mathcal{P}_{k,N} + \mathcal{Q}_{k,N})\eta(t) - \left(\delta^{2} + \frac{\overline{\mu}^{2}}{4}\right) \mathcal{I}_{k,N} v(t) -\\&- \Lambda^{2}(\mathcal{P}_{k,N} + \mathcal{Q}_{k,N})v(t).
	\end{split}
\end{equation}

Clearly, the $C_{0}$-semigroup generated by $A$ admits exponential dichotomy of rank $N$ such that the unstable subspace $\mathbb{H}^{u}_{A}$ is spanned by $e_{1},\ldots,e_{N}$ and the stable space $\mathbb{H}^{s}_{A}$ is spanned by the remained basis vectors $e_{N+1},e_{N+2},\ldots$. Since $A$ is self-adjoint, for $A^{*}=A$ we have the same subspaces. Then for the operator $\breve{A}$ the corresponding subspaces $\breve{\mathbb{H}}^{s}$ and $\breve{\mathbb{H}}^{u}$ are given by \eqref{EQ: StatDualCouplingStableUnstableSubspaces}.

Note that for $A$ (resp. $-A$), the corresponding semigroups satisfy \eqref{EQ: ExponentialDecayStationaryEstimate} with $M=1$ and $\varepsilon$ being the distance from $0$ to the spectrum of $A$ (resp. $-A$), i.e. $\overline{\mu} = - \lambda_{N} + \alpha = \lambda_{N+1} - \alpha$. Moreover, the same holds for the projected operator $A\mathcal{I}_{k,N}$ (resp. $-A\mathcal{I}_{k,N}$). As for $A (\mathcal{P}_{k,N} + \mathcal{Q}_{k,N})$ (resp. $-A (\mathcal{P}_{k,N} + \mathcal{Q}_{k,N})$), here the distance can estimated from below by $\overline{\mu} + k$ due to \eqref{EQ: SpatAvgProjectorsDef}. In particular, Theorem \ref{TH: LyapunovPerronOperatorStationaryTheorem} gives
\begin{equation}
	\label{EQ: LyapPerronOpsEstimateSelfAdj}
	\begin{split}
		\|\mathcal{I}_{k,N}\mathfrak{L}_{A}\| = \| \mathfrak{L}_{A} \| = \|\mathfrak{L}_{-A} \| \leq \frac{1}{\overline{\mu}},\\
		\|(\mathcal{P}_{k,N}+\mathcal{Q}_{k,N})\mathfrak{L}_{-A}\| = \|(\mathcal{P}_{k,N}+\mathcal{Q}_{k,N})\mathfrak{L}_{A}\| \leq \frac{1}{\overline{\mu}+k},
	\end{split}
\end{equation}
where $\| \cdot \|$ denotes the operator norm induced from $L_{2}(\mathbb{R};\mathbb{H})$.

For $q \in \mathcal{Q}$, let $\mathcal{M}^{q}_{a}$ be the operator in $L_{2}(0,\infty;\mathbb{H})$ given by
\begin{equation}
	(\mathcal{M}^{q}_{a}f)(t):= a(\vartheta^{t}(q))f(t) \text{ for } t \geq 0.
\end{equation}
for any $f \in L_{2}(0,\infty;\mathbb{H})$. We have the following lemma. Recall $\overline{\mu}:=(\lambda_{N+1}-\lambda_{N})/2 = -\lambda_{N} + \alpha$.
\begin{lemma}
	\label{LEM: SpatAvHamiltonSolvabilityVEquation}
	Suppose that $a_{b} := \sup_{q \in \mathcal{Q}} |a(q)| < \overline{\mu} + k$. Then for any $f(\cdot) \in L_{2}(0,\infty;\mathbb{H})$ and $q \in \mathcal{Q}$ there exists a unique mild solution $v(\cdot) \in L_{2}(0,\infty;\mathbb{H})$ on $[0,+\infty)$ to
	\begin{equation}
		\dot{v}(t) = Av(t) - a(\vartheta^{t}(q))(\mathcal{P}_{k,N} + \mathcal{Q}_{k,N})v(t) + f(t)
	\end{equation}
	such that $v(0)=v_{0}$. It satisfies the equation
	\begin{equation}
		\label{EQ: SpatAvEquationForV}
		v = \mathcal{R} \mathfrak{L}_{A} \mathcal{P} \left[ - (\mathcal{P}_{k,N} + \mathcal{Q}_{k,N})\mathcal{M}^{q}_{a}v + f \right]
	\end{equation}
	and admits the estimates
	\begin{equation}
		\label{EQ: SpatAvEquationForVEstimates}
		\begin{split}
			\| \mathcal{I}_{k,N} v(\cdot) \|_{L_{2}} &\leq \frac{1}{\overline{\mu}} \cdot \|\mathcal{I}_{k,N}f(\cdot) \|_{L_{2}},\\
			\| (\mathcal{P}_{k,N} + \mathcal{Q}_{k,N}) v(\cdot) \|_{L_{2}} &\leq \frac{1}{\overline{\mu} + k - a_{b}} \| (\mathcal{P}_{k,N}+ \mathcal{Q}_{k,N})f(\cdot) \|_{L_{2}},\\
		\end{split}
	\end{equation}
	where $L_{2}$ stands for $L_{2}(0,\infty;\mathbb{H})$. Moreover, analogous statement holds for the problem
	\begin{equation}
		\dot{\eta}(t) = -A\eta(t) + a(\vartheta^{t}(q))(\mathcal{P}_{k,N} + \mathcal{Q}_{k,N})\eta(t) + f(t),
	\end{equation}
	where the mild solution $\eta(\cdot) \in L_{2}(0,\infty;\mathbb{H})$ satisfies
	\begin{equation}
		\eta = \mathcal{R} \mathfrak{L}_{-A} \mathcal{P} \left[ (\mathcal{P}_{k,N} + \mathcal{Q}_{k,N}) \mathcal{M}^{q}_{a}\eta + f  \right].
	\end{equation}
\end{lemma}
\begin{proof}
	Clearly, any $v(\cdot)$ as in the statement must satisfy \eqref{EQ: SpatAvEquationForV} and vice versa. Let us solve \eqref{EQ: SpatAvEquationForV} by applying the projectors $\mathcal{I}_{k,N}$ and $\mathcal{P}_{k,N} + \mathcal{Q}_{k,N}$.
	
	By taking $\mathcal{I}_{k,N}$, we obtain
	\begin{equation}
		\mathcal{I}_{k,N} v = \mathcal{R} \mathfrak{L}_{A} \mathcal{P} \mathcal{I}_{k,N}f.
	\end{equation}
    From \eqref{EQ: LyapPerronOpsEstimateSelfAdj} we get the first estimate from \eqref{EQ: SpatAvEquationForVEstimates}.
	
	By taking $\mathcal{P}_{k,N} + \mathcal{Q}_{k,N}$, we obtain
	\begin{equation}
		(\mathcal{P}_{k,N} + \mathcal{Q}_{k,N}) v = \mathcal{R} \mathfrak{L}_{A} \mathcal{P} (\mathcal{P}_{k,N} + \mathcal{Q}_{k,N})\left[- \mathcal{M}^{q}_{a}  v + f \right].
	\end{equation}
	In virtue of \eqref{EQ: LyapPerronOpsEstimateSelfAdj}, the norm of $\mathcal{R} \mathfrak{L}_{A} \mathcal{P} (\mathcal{P}_{k,N} + \mathcal{Q}_{k,N})\mathcal{M}^{q}_{a}$ in $L_{2}(0,\infty;\mathbb{H})$ can be estimated from above by $a_{b}/(\overline{\mu} + k) < 1$. Consequently, the above equation is solvable for any $f$ and the second estimate is valid again due to \eqref{EQ: LyapPerronOpsEstimateSelfAdj}.
	
	For the second equation we have analogous arguments. The proof is finished.
\end{proof}

It is well-known that one may define the fractional power $A^{\beta}_{0}$ of $A_{0}$ for any $\beta \geq 0$ as follows. Let the domain $\mathbb{H}_{\beta} := \mathcal{D}(A^{\beta}_{0})$ of $A^{\beta}_{0}$ consist of such $v \in \mathbb{H}$ for which the sequence $\{\lambda^{\beta}_{j} \langle v, e_{j} \rangle_{\mathbb{H}} \}_{j \geq 1}$ is square-summable (in particular, $\mathbb{H}_{0} = \mathbb{H}$). Then we put
\begin{equation}
	A^{\beta} v := \sum_{j=1}^{\infty}\lambda^{\beta}_{j} \langle v, e_{j} \rangle_{\mathbb{H}} e_{j}
\end{equation}
for $v \in \mathbb{H}_{\beta}$. Clearly, the space $\mathbb{H}_{\beta}$ is naturally endowed with the inner product $\langle v , w\rangle_{\beta} := \langle A^{\beta}_{0}v, A^{\beta}_{0}w \rangle_{\mathbb{H}}$, where $v, w \in \mathbb{H}_{\beta}$, which makes it a Hilbert space. It is not hard to see that $A_{0}$ can be considered as a self-adjoint operator in $\mathbb{H}_{\beta}$ with the domain $\mathbb{H}_{1+\beta}$. It will be useful to apply the above and some of further constructions in this context also. 

For example, Lemma \ref{LEM: SpatAvHamiltonSolvabilityVEquation} is directly applicable for $\mathbb{H}$ exchanged with $\mathbb{H}_{\beta}$. Moreover, the operator $A$ in $\breve{\mathbb{H}} = \mathbb{H} \times \mathbb{H}$ can be also considered as an operator in $\breve{\mathbb{H}}_{\beta} := \mathbb{H}_{\beta} \times \mathbb{H}_{\beta}$ and it admits exponential dichotomy with a stable space $\breve{\mathbb{H}}^{s}_{\beta}$ and an unstable space $\breve{\mathbb{H}}^{u}_{\beta}$. Clearly, $\breve{\mathbb{H}}^{s}_{\beta_{1}} \subset \breve{\mathbb{H}}^{s}_{\beta_{2}}$ and $\breve{\mathbb{H}}^{u}_{\beta_{1}} \subset \breve{\mathbb{H}}^{u}_{\beta_{2}}$ for $\beta_{1} \geq \beta_{2} \geq 0$.

Recall the complex structure $J$ in $\mathbb{H}_{\beta}$ given by $J(v,\eta) := (-\eta,v)$ for any $(v,\eta) \in \breve{\mathbb{H}}_{\beta}$. It is important that the Hamiltonian $H(q)$ is well-defined\footnote{For more general problems, similar properties require additional study.} as an operator from $\breve{\mathbb{H}}_{1+\beta}$ to $\breve{\mathbb{H}}_{\beta}$ due to \eqref{EQ: SpatAvNonstatOptConcreteEquation}. Clearly,
\begin{equation}
	\label{EQ: SymplecticIdentitySpatAvrgScale}
	J H(q) + H^{*}(q)J = 0 \text{ on } \mathbb{H}_{1+\beta} \times \mathbb{H}_{1+\beta}.
\end{equation}

For a given $q \in \mathcal{Q}$, let $\mathbb{L}^{+}_{b}(q)$ be the subspace of all $z_{0} \in \breve{\mathbb{H}}$ such that there exists a solution $z(t) = z(t;q,z_{0})$ to \eqref{EQ: NonStatHamiltonianSystemSpatAv} on $[0,\infty)$ with $z(0) = z_{0}$. Analogously, for a given real $\varepsilon$ we define $\mathbb{L}^{+}_{\varepsilon}(q)$ as the subspace of all $z_{0} \in \breve{\mathbb{H}}$ such that $e^{\varepsilon t} z(t;q,z_{0})$ belongs to $L_{2}(0,\infty;\breve{\mathbb{H}})$. Moreover, the same definitions can be applied by exchanging the space $\mathbb{H}$ with $\mathbb{H}_{\beta}$ for any $\beta \geq 0$. This leads to the subspaces $\mathbb{L}^{+}_{b,\beta}(q)$ and $\mathbb{L}^{+}_{\varepsilon,\beta}(q)$ in $\breve{\mathbb{H}}_{\beta}$.

Now we are ready to show the existence of stable Lagrangian bundles for \eqref{EQ: NonStatHamiltonianSystemSpatAv} under conditions of the Spatial Averaging Principle (see Remark \ref{REM: ImprovedSpatAvrgConds}).
\begin{theorem}
	\label{TH: SpatAvrgLagrangeBundle}
	Suppose there exist positive integers $k$ and $N$ and positive reals $\Lambda$ and $\delta$ such that (recall $\overline{\mu} = (\lambda_{N+1} - \lambda_{N}) / 2$)
	\begin{equation}
		\label{EQ: SpatAvgLagrangeBundlesInequaltiies}
		\frac{\overline{\mu}}{\sqrt{5}} - \delta \geq 0 \text{ and } k^{2} - 2\Lambda k - 4 \frac{\Lambda^{4}}{\overline{\mu}^{2}} \geq 0.
	\end{equation}
	Then for the Hamiltonian system \eqref{EQ: NonStatHamiltonianSystemSpatAv} satisfying \eqref{EQ: SpatAvrgTauChoice} and \eqref{EQ: SpatAvgFunctionABound}, $\mathbb{L}^{+}_{b,\beta}(q)$ is a Lagrange subspace in $\breve{\mathbb{H}}_{\beta}$ which coincides with the subspace $\mathbb{L}^{+}_{\varepsilon,\beta}(q)$ for any $q \in \mathcal{Q}$ and $|\varepsilon| \leq \varepsilon_{0}$ with some $\varepsilon_{0}>0$ independent of $q$ and $\beta \geq 0$. Moreover, for any $0 < \varepsilon \leq \varepsilon_{0}$ there exists a constant $M_{\varepsilon} > 0$ such that the estimate
	\begin{equation}
		\label{EQ: SpatAveragUniformExpDecay}
		|z(t;q,z_{0})|_{\breve{\mathbb{H}}_{\beta}} \leq M_{\varepsilon} e^{-\varepsilon t} |z_{0}|_{\breve{\mathbb{H}}_{\beta}}
	\end{equation}
	holds for any $q \in \mathcal{Q}$, $\beta \geq 0$, $z_{0} \in \mathbb{L}^{+}_{b,\beta}(q)$ and the corresponding solution $z(t;q,z_{0})$ with $z(0;q,z_{0}) = z_{0}$.
	
	In addition, $\mathbb{L}^{+}_{b,\beta}(q)$ is given by the graph of a bounded linear operator $M^{+}_{\beta}(q) \colon \breve{\mathbb{H}}^{s}_{\beta} \to \breve{\mathbb{H}}^{u}_{\beta}$ as
	\begin{equation}
		\label{EQ: SpatAvgGraphFormula}
		\mathbb{L}^{+}_{b,\beta}(q) = \{ z^{s} +  M^{+}_{\beta}(q)z^{s} \ | \ z^{s} \in \breve{\mathbb{H}}^{s}\},
	\end{equation}
	where the norms of $M^{+}_{\beta}(q)$ are bounded uniformly in $q \in \mathcal{Q}$ and $\beta \geq 0$ and the mapping $\mathcal{Q} \ni q \mapsto M^{+}_{\beta}(q)$ is continuous in the operator norm for any $\beta \geq 0$.
\end{theorem}
\begin{proof}
	We will firstly show \eqref{EQ: SpatAveragUniformExpDecay} and \eqref{EQ: SpatAvgGraphFormula} for $\beta = 0$.
	
	Let us fix $q \in \mathcal{Q}$ and involve Lyapunov-Perron operators into the problem as in the first part of the proof of Theorem \ref{TH: DichotomyStatOptimizationLP}. Namely, take $z^{s} \in \breve{\mathbb{H}}^{s}$ and let us find a solution $z(\cdot) \in L_{2}(0,\infty;\breve{\mathbb{H}})$ to \eqref{EQ: NonStatHamiltonianSystemSpatAv} such that $\Pi^{s}_{\breve{A}} z(0) = z^{s}$. Suppose for a moment that we have already found such $z(\cdot)$. Now consider $g(\cdot) = g(\cdot;z^{s})$ such that $g(t)  := G_{\sharp}(t)z^{s}$ for $t>0$ and let $R(\vartheta^{t}(q))g(t) = (g_{v}(t),g_{\eta}(t))$ for any $t \geq 0$, where $R(q)$ is given by \eqref{EQ: NonStatHamiltonianSpatAvPerturbationMatrixR}. Then $\Delta z(t) = (\Delta v(t), \Delta \eta(t)) := z(t) - G_{\sharp}(t)z^{s}$, according to \eqref{EQ: SpatAvNonstatOptConcreteEquation}, satisfies
	\begin{equation}
		\label{EQ: SpatAvgSystemDeltas}
		\begin{split}
			&\frac{d}{dt}\Delta v(t) = A \Delta v(t) - a(\vartheta^{t}(q))(\mathcal{P}_{k,N} + \mathcal{Q}_{k,N})\Delta v(t) + \mathcal{I}_{k,N}\Delta\eta(t) +\\&+ \Delta\eta(t) + 4\left(\frac{\Lambda}{\overline{\mu}}\right)^{2}(\mathcal{P}_{k,N} + \mathcal{Q}_{k,N}z)\Delta\eta(t) + g_{v}(t),\\
			&\frac{d}{dt}\Delta \eta(t) = -A\Delta \eta(t)  + a(\vartheta^{t}(q))(\mathcal{P}_{k,N} + \mathcal{Q}_{k,N})\Delta\eta(t) -\\&- \left(\delta^{2} + \frac{\overline{\mu}^{2}}{4}\right) \mathcal{I}_{k,N} \Delta v(t) - \Lambda^{2}(\mathcal{P}_{k,N} + \mathcal{Q}_{k,N})\Delta v(t) + g_{\eta}(t).
		\end{split}
	\end{equation}
	
	Let $\mathcal{T}$ be the operator which takes each $f(\cdot) \in L_{2}(0,\infty;\mathbb{H})$ into the unique $\eta(\cdot) \in L_{2}(0,\infty;\mathbb{H})$ satisfying \eqref{EQ: SpatAvNonstatOptConcreteEquation}, where $\eta(t)$ in the fist equation is exchanged with $f(t)$. It is a well-defined bounded operator by Lemma \ref{LEM: SpatAvHamiltonSolvabilityVEquation}.
	
	It is not hard to see that \eqref{EQ: SpatAvgSystemDeltas} is equivalent to the solvability of
	\begin{equation}
		\label{EQ: SpatAvgBundleResolutionEq}
		\Delta \eta = \mathcal{T} \Delta \eta + \mathcal{T}_{0} \begin{pmatrix}
			g_{v}\\
			g_{\eta}
		\end{pmatrix}
	\end{equation}
	for certain bounded operator $\mathcal{T}_{0}$ from $L_{2}(0,\infty;\breve{\mathbb{H}})$ to $L_{2}(0,\infty;\mathbb{H})$.
	
	Let us show that $\mathcal{T}$ is a uniform (in $q \in \mathcal{Q}$) contraction in $L_{2}(0,\infty;\mathbb{H})$. From Lemma \ref{LEM: SpatAvHamiltonSolvabilityVEquation} and the first inequality from \eqref{EQ: SpatAvgLagrangeBundlesInequaltiies} we obtain
	\begin{equation}
		\|\mathcal{I}_{k,N}\mathcal{T}\| \leq \frac{(1+1)}{\overline{\mu}^{2}} \left(\delta^{2} + \frac{\overline{\mu}^{2}}{4}\right) = \frac{1}{2} + \frac{2 \delta^{2}}{\overline{\mu}^{2}} \leq \frac{1}{2} + \frac{2}{5} < 1.
	\end{equation}
	Analogously, from Lemma \ref{LEM: SpatAvHamiltonSolvabilityVEquation} and the second inequality from \eqref{EQ: SpatAvgLagrangeBundlesInequaltiies} we get
	\begin{equation}
		\|(\mathcal{P}_{k,N}+\mathcal{Q}_{k,N})\mathcal{T}\| \leq \left(1 + 4 \left(\frac{\Lambda}{\overline{\mu}}\right)^{2} \right) \cdot \frac{\Lambda^{2}}{(\overline{\mu} + k - a_{b})^{2}}.
	\end{equation}
	From \eqref{EQ: SpatAvgFunctionABound} and the first inequality from \eqref{EQ: SpatAvgLagrangeBundlesInequaltiies} we have $a_{b} \leq \Lambda + \overline{\mu}/\sqrt{5}$. Note that the second inequality from \eqref{EQ: SpatAvgLagrangeBundlesInequaltiies} particularly gives that $k > 2\Lambda$. Consequently, 
	\begin{equation}
		(\overline{\mu} + k - a_{b})^{2} \geq (\overline{\mu} + k - \Lambda - \overline{\mu}/\sqrt{5})^{2} > (k - \Lambda)^{2}.
	\end{equation}
	Thus, for the contraction it is sufficient to verify that
	\begin{equation}
		4 \frac{\Lambda^{4}}{\overline{\mu}^{2}} + \Lambda^{2} \leq (k - \Lambda)^{2}.
	\end{equation}
	that is obviously equivalent to the second inequality from \eqref{EQ: SpatAvgLagrangeBundlesInequaltiies}.
	
	Thus, $\mathcal{T}$ is a uniform contraction and, consequently, \eqref{EQ: SpatAvgBundleResolutionEq} has a unique solution $\Delta\eta \in L_{2}(0,\infty;\mathbb{H})$ for each $g$. As in the proof of Theorem \ref{TH: DichotomyStatOptimizationLP}, from Theorem \ref{TH: LyapunovPerronOperatorStationaryTheorem} we get that $M^{+}(q)z^{s} := (\Delta v(0), \Delta \eta(0))$ is a bounded linear operator and in our case the norms are bounded uniformly in $q \in \mathcal{Q}$.
	
	To show that there exists $\varepsilon_{0} > 0$ such that $\mathbb{L}^{+}_{b}(q) = \mathbb{L}^{+}_{\varepsilon}(q)$ for any $q \in \mathcal{Q}$ and $|\varepsilon| < \varepsilon_{0}$, we act as in the proof of Theorem \ref{TH: DichotomyStatOptimizationLP} by exchanging $A$ with $A+\varepsilon I$ and noticing that for small $\varepsilon$ the above arguments can be also applied since the norms of Lyapunov-Perron operators change slightly.
	
	Using the boundedness of the Lyapunov-Perron operator from $L_{2}(\mathbb{R};\breve{\mathbb{H}})$ into $C_{b}(\mathbb{R};\breve{\mathbb{H}})$ given by Theorem \ref{TH: LyapunovPerronOperatorStationaryTheorem}, we get that \eqref{EQ: SpatAveragUniformExpDecay} is satisfied with
	\begin{equation}
		M_{\varepsilon}:=\sup_{t \geq 0}\sup_{q \in \mathcal{Q}}\sup_{|z_{0}|_{\breve{\mathbb{H}}} \leq 1}| e^{\varepsilon t}z(t;q,z_{0}) |_{\breve{\mathbb{H}}} < \infty.
	\end{equation}

    Consequently, for any $t \geq 0$ and $q \in \mathcal{Q}$ we have
	\begin{equation}
		|z(t;z^{s},q)| \leq M_{\varepsilon} e^{-\varepsilon t} 
	\end{equation}
	
	Now let us show that the mapping $\mathcal{Q} \ni q \mapsto M^{+}(q)$ is continuous in the operator norm. For $q \in \mathcal{Q}$ and $z^{s} \in \breve{\mathbb{H}}^{s}$ we put $\Delta z(t;q,z^{s}) := z(t;q,s^{s} + M^{+}(q)z^{s}) - G_{\sharp}(t)z^{s}$ for $t \geq 0$. By definition, we have $M^{+}(q)z^{s} = \Delta z(0;q,z^{s})$. Suppose a sequence $q_{m}$, where $m=1,2,\ldots$, tends to some $q \in \mathcal{Q}$ as $m \to \infty$. Then $(\Delta v_{m}(t), \Delta \eta_{m}(t)) := \Delta z(t;q_{m},z^{s}) - \Delta z(t;q,z^{s}) = $ satisfies an analog of \eqref{EQ: SpatAvgBundleResolutionEq} as
	\begin{equation}
		\Delta \eta_{m} = \mathcal{T}\Delta \eta_{m} + \mathcal{T}_{0}g_{m},
	\end{equation}
	where $g_{m}(t) := [a(\vartheta^{t}(q_{m})) - a(\vartheta^{t}(q))]z(t;q,z^{s}+M^{+}(q)z^{s})$ for $t \geq 0$. We have already shown that the equation is uniquely solvable and $\Delta\eta_{m}$ continuously depend on $g_{m}$ in appropriate $L_{2}$-spaces with a uniform bound. From \eqref{EQ: SpatAveragUniformExpDecay} and the continuity of $a(\cdot)$, it is clear that $g_{m}$ tends to $0$ in $L_{2}(0,\infty;\breve{\mathbb{H}})$. Consequently, $\Delta\eta_{m}$ tends to $0$ in $L_{2}(0,\infty;\mathbb{H})$ and the same holds for $\Delta v_{m}$. From this and since the Lyapunov-Perron operator $\mathfrak{L}_{\breve{A}}$ is bounded from $L_{2}(\mathbb{R};\breve{\mathbb{H}})$ to $C_{b}(\mathbb{R};\breve{\mathbb{H}})$ (see Theorem \ref{TH: LyapunovPerronOperatorStationaryTheorem}), we get that (use \eqref{EQ: NonStatHamiltonianSystemSpatAvRewritten}) $M^{+}(q_{m})z^{s} - M^{+}(q)z^{s} = (\Delta v_{m}(0), \Delta \eta_{m}(0))$ tends to $0$ as $m \to \infty$ uniformly in $|z^{s}|_{\breve{\mathbb{H}}} = 1$ that shows the required continuity.
	
	Note that the above arguments are based on the estimates from Lemma \ref{LEM: SpatAvHamiltonSolvabilityVEquation} and Theorem \ref{TH: LyapunovPerronOperatorStationaryTheorem} which depend only on the spectrum of the associated operators and, consequently, the same estimates remain valid when $\mathbb{H}$ is exchanged with $\mathbb{H}_{\beta}$ for any $\beta \geq 0$. This not only shows that the above arguments directly applicable for any $\beta \geq 0$, but also that the arising constants are uniform in $\beta$.
	
	Clearly, $M^{+}_{\beta_{1}}(q)$ coincides with $M^{+}_{\beta_{2}}(q)$ on $\breve{\mathbb{H}}^{s}_{\beta_{1}}$ for any $\beta_{1} \geq \beta_{2} \geq 0$. From this, \eqref{EQ: SpatAvgGraphFormula} and since $\breve{\mathbb{H}}^{s}_{\beta_{1}}$ is dense in $\breve{\mathbb{H}}^{s}_{\beta_{2}}$, we get that $\mathbb{L}^{+}_{b,\beta_{1}}(q)$ is dense in $\mathbb{L}^{+}_{b,\beta_{2}}(q)$ for any $q \in \mathcal{Q}$ and $\beta_{1} \geq \beta_{2} \geq 0$.
	
	It remains to show that $\mathbb{L}^{+}_{b,\beta}(q)$ is Lagrange and this can be done similarly to the argument applied around \eqref{EQ: StatDichPertbIsotropicSympIden}. We have that for any $z_{0} \in \mathbb{L}^{+}_{b,\gamma}(q)$ with $\gamma > 0$, the corresponding mild solution $z(t)=z(t;q,z_{0})$ belongs to $C([s,\infty);\breve{\mathbb{H}}_{1+\beta}) \cap C^{1}([s,\infty);\breve{\mathbb{H}}_{\beta})$ for any $s > 0$ and $\beta \in [0,\gamma)$ and satisfies \eqref{EQ: NonStatHamiltonianSystemSpatAv} for  $t > 0$ due to the parabolic smoothing (see, for example, Theorem 6.7, Chapter I in \cite{Krein1971}). From this, differentiating and using \eqref{EQ: SymplecticIdentitySpatAvrgScale}, for any $z^{1}_{0}, z^{2}_{0} \in \mathbb{L}^{+}_{b,\gamma}(q)$ we have
	\begin{equation}
		\label{EQ: SymplecticPreservationSpatAvrg}
		\langle z_{1}(t) , J z_{2}(t) \rangle_{\beta} = \langle z^{1}_{0}, Jz^{2}_{0} \rangle_{\beta} \text{ for any } t \geq 0 \text{ and } \beta \in [0,\gamma),
	\end{equation}
	where $z_{1}(t)=z(t;q,z^{1}_{0})$ and $z_{2}(t) = z(t;q,z^{2}_{0})$ are the corresponding mild solutions. Since $\mathbb{L}^{+}_{b,\gamma}(q)$ is dense in $\mathbb{L}^{+}_{b,\beta}(q)$, the above identity can be extended by continuity to any $z^{1}_{0}, z^{2}_{0} \in \mathbb{L}^{+}_{b,\beta}(q)$. Passing to the limit in \eqref{EQ: SymplecticPreservationSpatAvrg} as $t \to +\infty$ and using \eqref{EQ: SpatAveragUniformExpDecay}, we get that that $\mathbb{L}^{+}_{b,\beta}(q)$ is isotropic. Then Theorem \ref{TH: LagrangeSubspaceCriterion} and \eqref{EQ: SpatAvgGraphFormula} imply that $\mathbb{L}^{+}_{b,\beta}(q)$ must be Lagrange. Choosing $\gamma$ arbitrarily large, we get this for any $\beta \geq 0$. The proof is finished.
\end{proof}

\begin{remark}
	\label{REM: ImprovedSpatAvrgConds}
	It is not hard to see that the inequalities from \eqref{EQ: SpatAvgLagrangeBundlesInequaltiies} are sharper than the inequalities from Theorem 4.2 in \cite{KostiankoZelikSA2020} which can be described in our case as
	\begin{equation}
		\label{EQ: SpatAvrgInequalitiesZelik}
		\frac{\overline{\mu}}{4} - \delta > 0, \ \frac{1}{2} k - \frac{16\Lambda^{2}}{\overline{\mu}} - 2 \Lambda \geq 0.
	\end{equation}
	In the next section, it will be shown that \eqref{EQ: SpatAvrgInequalitiesZelik} implies the nonoscillation.
\end{remark}

\subsection{Nonoscillation}

Now we are interested in whether there exists a field of bounded self-adjoint operators $P \colon \mathcal{Q} \to \mathcal{L}(\mathbb{H})$ such that
\begin{equation}
	\label{EQ: NonoscillationSpatAvrg}
	\mathbb{L}^{+}_{b}(q) = \{ (v,-P(q)v) \ | \ v \in \mathbb{H} \} \text{ for any } q \in \mathcal{Q}.
\end{equation}

Note that we automatically have that $P(q)$ must depend continuously on $q \in \mathcal{Q}$ as follows.
\begin{lemma}
	Under the conditions of Theorem \ref{TH: SpatAvrgLagrangeBundle}, suppose that \eqref{EQ: NonoscillationSpatAvrg} is satisfied. Then $P(q)$ depend continuously on $q \in \mathcal{Q}$ in the operator norm.
\end{lemma}
\begin{proof}
	Using \eqref{EQ: SpatAvgGraphFormula} and the continuity of $\mathcal{Q} \ni q \mapsto M^{+}(q)$, we can apply Theorem \ref{PROP: LagrangeSubspacesContinuousDependence} with $\mathbb{H}_{\sharp} := \breve{\mathbb{H}}^{s}$, $\mathbb{H}_{\flat}:=\breve{\mathbb{H}}^{u}$ and $M(q) := M^{+}(q)$ to get that $\mathcal{Q} \ni q \mapsto \mathbb{L}^{+}_{b}(q) \in \Lambda(\mathbb{H})$ is continuous. Now under \eqref{EQ: NonoscillationSpatAvrg}, we apply the same theorem with $\mathbb{H}_{\sharp} := \mathbb{L}_{v} = \mathbb{H} \times \{0_{\mathbb{H}}\}$, $\mathbb{H}_{\flat} := \mathbb{L}_{\eta} = \{0_{\mathbb{H}}\} \times \mathbb{H}$ and $M(q) := P(q)$, to get the continuity of $P(q)$ in $q \in \mathcal{Q}$. The proof is finished.
\end{proof}

Under \eqref{EQ: NonoscillationSpatAvrg} one says that the corresponding Hamiltonian system \eqref{EQ: NonStatHamiltonianSystemSpatAv} (or the Lagrange bundle with the fiber $\mathbb{L}^{+}_{b}(q)$) over $(\mathcal{Q},\vartheta)$ is nonoscillating. If the norms of $P(q)$ are bounded uniformly in $q \in \mathcal{Q}$, one speaks of a \textit{uniform nonoscillation}.

Recall the ``vertical'' Lagrange subspace $\mathbb{L}_{\eta} := \{0_{\mathbb{H}}\} \times \mathbb{H}$. From the construction of $\mathbb{L}^{+}_{b}(q)$, similarly to Corollary \ref{COR: StatDichLagrangeSubsFredholmPair} we obtain the following.
\begin{corollary}
	\label{COR: SpatAvrgLagrangeSubsFredholmPair}
	Under the conditions of Theorem \ref{TH: SpatAvrgLagrangeBundle}, for any $q \in \mathcal{Q}$ we have the following:
	\begin{enumerate}
		\item[1)] $\dim (\mathbb{L}^{+}_{b}(q) \cap \mathbb{L}_{\eta}) \leq N$;
		\item[2)] $\mathbb{L}^{+}_{b}(q) + \mathbb{L}_{\eta}$ is closed and has finite codimension $\leq N$.
	\end{enumerate}
	In particular, $(\mathbb{L}^{+}_{b}(q), \mathbb{L}_{\eta})$ is a Fredholm pair (see \cite{Furutani2004FLG}).
\end{corollary}

From this, similarly to Proposition \ref{PROP: NonoscillationStationaryAsTrivialIntersection}, we obtain.
\begin{proposition}
	\label{PROP: NonoscillationSpatAvrgAsTrivialIntersection}
	Under the conditions of Theorem \ref{TH: SpatAvrgLagrangeBundle}, \eqref{EQ: NonoscillationSpatAvrg} holds if and only if $\mathbb{L}^{+}_{b}(q)$ intersects $\mathbb{L}_{\eta}$ trivially for any $q \in \mathcal{Q}$.
\end{proposition}

To proceed further, we need to establish relations between \eqref{EQ: NonStatHamiltonianSystemSpatAv} and quadratic optimization. For this, consider the control system in $\mathbb{H}$ given by
\begin{equation}
	\label{EQ: SpatAvrgControlSystemNonoscExpl}
	\dot{v}(t) = A(\vartheta^{t}(q))v(t) + B\xi(t),
\end{equation}
where $A(q)$ is defined above \eqref{EQ: NonStatHamiltonianSystemSpatAv}, i.e. $A(q) = -A_{0} + \nu_{0}(q) I$ with $\nu_{0}(q) = \alpha - a(q)$ and $\alpha = (\lambda_{N}+\lambda_{N+1})/2$; $\xi(\cdot) \in L_{2}(0,T;\mathbb{U})$ for some $T>0$ and $\mathbb{U} = \mathbb{H} \times \mathbb{H}$; $B\xi := \xi_{I} + \xi_{c}$ for $\xi=(\xi_{I},\xi_{c}) \in \mathbb{U}$.

Similarly to \eqref{EQ: StatHamiltSpaceOfProcessesNorm}, for any $q \in \mathcal{Q}$, we define the set $\mathfrak{M}_{v_{0}}(q)$ of processes through $v_{0} \in \mathbb{H}$ over $q$ and the space of processes $\mathcal{Z}(q)$ given by the union of all $\mathfrak{M}_{v_{0}}(q)$ over $v_{0} \in \mathbb{H}$ and equipped with the norm \eqref{EQ: StatHamiltSpaceOfProcessesNorm} which makes it a Banach space.

For any $q \in \mathcal{Q}$, there is a well-defined integral quadratic functional on $\mathcal{Z}(q)$ given by
\begin{equation}
	\label{EQ: IntegralQuadraticFunctionalSpatAvrg}
	\mathcal{J}_{q}(v(\cdot),\xi(\cdot)) := \int_{0}^{+\infty}\mathcal{F}_{\vartheta^{t}(q)}(v(t),\xi(t))dt,
\end{equation}
where $\mathcal{F}_{q}$ is given by \eqref{EQ: QuadraticFormSpatialAveraging}.

\begin{theorem}
	Under the conditions of Theorem \ref{TH: SpatAvrgLagrangeBundle}, suppose that $\mathcal{J}_{q}$ admits a unique critical point on $\mathfrak{M}_{0}(q)$ for any $q \in \mathcal{Q}$. Then the nonoscillation condition \eqref{EQ: NonoscillationSpatAvrg} is satisfied.
\end{theorem}
\begin{proof}
	In virtue of Proposition \ref{PROP: NonoscillationSpatAvrgAsTrivialIntersection}, it is sufficient to show that $\mathbb{L}^{+}_{b}(q)$ intersects $\mathbb{L}_{\eta}$ trivially for any $q \in \mathcal{Q}$. Supposing the opposite, we obtain $(0,\eta_{0}) \in \mathbb{L}^{+}_{b}(q)$ with some nonzero $\eta_{0} \in \mathbb{H}$. Consider the corresponding solution $(v(\cdot),\eta(\cdot))$ of \eqref{EQ: NonStatHamiltonianSystemSpatAv} with $v(0) = 0$ and $\eta(0) = \eta_{0}$ and put $\xi(t) := -F^{-1}_{3}F_{2}(\vartheta^{t}(q))v(t) + F^{-1}_{3}B^{*}\eta(t)$ for $t \geq 0$. Clearly, $(v(\cdot),\xi(\cdot))$ belongs to $\mathfrak{M}_{0}(q)$. In the same way as in Lemma \ref{LEM: StationaryStabLagrangeCritPoints}, we may show that it must be a critical point for $\mathcal{J}_{q}$. Since such a point is unique by the hypotheses, we must have $v(\cdot) \equiv 0$ and $\xi(\cdot) \equiv 0$. In particular, $B^{*}\eta(t) = 0$ and, due to \eqref{EQ: NonStatHamiltonianSystemSpatAv}, we get that $\eta(\cdot)$ must satisfy $\dot{\eta}(t) = - A\eta(t) + (\mathcal{P}_{k,N} + \mathcal{Q}_{k,N})\eta(t)$ for all $t \geq 0$ with $A$ as in \eqref{EQ: NonStatHamiltonianSystemSpatAvRewritten}. Since $\sup_{q \in \mathcal{Q}}|a(q)| < k + \overline{\mu}$, the initial data $\eta_{0}$ must belong to the unstable subspace $\mathbb{H}^{u}_{A}$ of $A$, i.e. the linear span of $e_{1}, \ldots, e_{N}$. From this, by repeating the proof of Proposition \ref{PROP: NonoscillationUniqueFixedPoint}, we deduce that $B$ must vanish on a subspace of $\mathbb{H}^{u}_{A}$ that contradicts to the $L_{2}$-controllability of $(A,B)$. The proof is finished.
\end{proof}

Now along with the projectors from \eqref{EQ: SpatAvgProjectorsDef}, we consider the complementary projectors $\mathcal{P}_{N}$ and $\mathcal{Q}_{N}$, where $\mathcal{P}_{N}$ is the orthogonal projector onto the linear span of $e_{1},\ldots,e_{N}$. Let us consider the quadratic form in $\mathbb{H}$ given by
\begin{equation}
	\label{EQ: SingularQuadraticFormSpatAvrg}
	V(v) :=  \frac{\overline{\mu}}{2} \cdot|\mathcal{P}_{N} v|^{2}_{\mathbb{H}}  - \frac{\overline{\mu}}{2} \cdot |\mathcal{Q}_{N} v|^{2}_{\mathbb{H}} \text{ for } v \in \mathbb{H}.  
\end{equation}

The following theorem gives a view on Theorem 4.2 in \cite{KostiankoZelikSA2020} (with $\gamma = 0$) as a result which investigates the nonoscillation (see Theorem \ref{TH: SpatAveragingNonOscill} below). Although there is a difference in the statements (and, in fact, our statement is stronger for the considered case), the proof follows similar lines.
\begin{theorem}
	\label{TH: SpatAvrgAuxiliaryCone}
	Suppose there exist positive integers $k$ and $N$ and positive reals $\Lambda$ and $\delta$ such that the inequalities 
	\begin{equation}
		\label{EQ: SpatAvrgNonOscillConditions}
		\frac{\overline{\mu}}{2} - \delta > 0 \text{ and } k - \frac{5\Lambda^{2}}{\overline{\mu}} - \Lambda \geq 0
	\end{equation}
	 are satisfied. Let the coefficients $\tau_{1}$, $\tau_{2}$ and $\tau_{3}$ be as in \eqref{EQ: SpatAvrgTauChoice}, i.e. $\tau_{1} = 1$, $\tau_{2} = \overline{\mu}^{2}/(4\Lambda^{2})$ and $\tau_{3} = 1$. Then there exists $\delta_{V} > 0$ such that the quadratic form $V(\cdot)$ from \eqref{EQ: SingularQuadraticFormSpatAvrg} satisfies
	\begin{equation}
		\label{EQ: SpatAvrgSingFuncIntegralCond}
		V(v(T)) - V(v_{0}) + \int_{0}^{T}\mathcal{F}_{\vartheta^{t}(q)}(v(t),\xi(t))dt \geq \delta_{V} \int_{0}^{T}(|v(t)|^{2}_{\mathbb{H}} + |\xi(t)|^{2}_{\mathbb{U}})dt
	\end{equation}
	for any $q \in \mathcal{Q}$, $T>0$ and any solution $v(\cdot)$ to \eqref{EQ: SpatAvrgControlSystemNonoscExpl} on $[0,T]$ with $\xi(\cdot) \in L_{2}(0,T;\mathbb{U})$.
\end{theorem}
\begin{proof}
	For $\varepsilon>0$, let us put 
	\begin{equation}
		\mathcal{F}^{(\varepsilon)}_{q}(v,\xi) :=  (\tau_{1}-\varepsilon) \mathcal{F}_{1,q}(v,\xi) + (\tau_{2}-\varepsilon)\mathcal{F}_{2}(v,\xi) + (\tau_{3}-\varepsilon) \mathcal{F}_{3}(v,\xi),
	\end{equation}
	where the quadratic forms on the right-hand side are given by \eqref{EQ: SpatAvrgQuadraticForms}.
	
	Note that it is sufficient to show that there exists $\delta_{0}>0$ such that for all sufficiently small $\varepsilon > 0$ we have
	\begin{equation}
		\langle Av - a(q)v + \xi_{I} + \xi_{c}, \overline{\mu}(\mathcal{P}_{N} - \mathcal{Q}_{N})v \rangle_{\mathbb{H}} + \mathcal{F}^{(\varepsilon)}_{q}(v,\xi) \geq \delta_{0}|v|^{2}_{\mathbb{H}}
	\end{equation}
	satisfied for any $q \in \mathcal{Q}$, $v \in \mathcal{D}(A) = \mathcal{D}(A_{0})$ and $\xi = (\xi_{I},\xi_{c}) \in \mathbb{U}$. Indeed, by taking $\varepsilon>0$ sufficiently small, from the above inequality and \eqref{EQ: SpatAvrgQuadraticForms} we get the infinitesimal analog of \eqref{EQ: SpatAvrgSingFuncIntegralCond}
	\begin{equation}
		\langle Av - a(q)v + \xi_{I} + \xi_{c}, \overline{\mu}(\mathcal{P}_{N} - \mathcal{Q}_{N})v \rangle_{\mathbb{H}} + \mathcal{F}_{q}(v,\xi) \geq \delta_{V} (|v|^{2}_{\mathbb{H}} + |\xi|^{2}_{\mathbb{U}})
	\end{equation}
	with an appropriate $\delta_{V}>0$.
	
	By (4.6) from \cite{KostiankoZelikSA2020}, we have
	\begin{equation}
		\begin{split}
			\langle A\mathcal{P}_{k,N}v,\mathcal{P}_{k,N} v \rangle_{\mathbb{H}} \geq (\overline{\mu}+k) | \mathcal{P}_{k,N} v |^{2}_{\mathbb{H}},\\
			\langle -A\mathcal{Q}_{k,N}v,\mathcal{Q}_{k,N} v \rangle_{\mathbb{H}} \geq (\overline{\mu}+k) | \mathcal{Q}_{k,N} v |^{2}_{\mathbb{H}}.
		\end{split}
	\end{equation}
	Moreover, similarly to the estimates from Proposition 4.1 in \cite{KostiankoZelikSA2020}, we get
	\begin{equation}
		\begin{split}
			\langle A\mathcal{P}_{N} \mathcal{I}_{k,N}v, \mathcal{P}_{N}\mathcal{I}_{k,N}v \rangle_{\mathbb{H}} \geq \overline{\mu} |\mathcal{P}_{N}\mathcal{I}_{k,N} v|^{2}_{\mathbb{H}},\\
			\langle -A\mathcal{Q}_{N} \mathcal{I}_{k,N}v, \mathcal{Q}_{N}\mathcal{I}_{k,N}v \rangle_{\mathbb{H}} \geq \overline{\mu} |\mathcal{Q}_{N}\mathcal{I}_{k,N} v|^{2}_{\mathbb{H}}.
		\end{split}
	\end{equation}
	Combining these inequalities and using that $\mathcal{P}_{N} v = \mathcal{P}_{k,N}v + \mathcal{P}_{N}\mathcal{I}_{k,N} v$ and $\mathcal{Q}_{N}v = \mathcal{Q}_{k,N}v + \mathcal{Q}_{N}\mathcal{I}_{k,N} v$, we get
	\begin{equation}
		\label{EQ: SpatAvrgNonOscLinearPart}
		\langle Av, \overline{\mu}(\mathcal{P}_{N} - \mathcal{Q}_{N})v \rangle_{\mathbb{H}} \geq \overline{\mu} k | (\mathcal{P}_{k,N} + \mathcal{Q}_{k,N}) v|^{2}_{\mathbb{H}} + \overline{\mu}^{2}| v |^{2}_{\mathbb{H}}.
	\end{equation}
	
	From the Cauchy and Young inequalities\footnote{Below, we are often aimed to estimate a product $ab$ with $a,b \geq 0$ from above as $\gamma a^{2} + \ldots$ for a given $\gamma > 0$. For this, we write $a b = 2 \cdot (\sqrt{\gamma} a) (b/(2\sqrt{\gamma}))$ and apply the Young inequality to the brackets. Thus, $a b \leq \gamma a^{2} + b^{2}/(4\gamma)$.}, we get
	\begin{equation}
		\label{EQ: SpatAvrgNonOscXiCPart}
		\begin{split}
			\langle \xi_{c}, \overline{\mu}(\mathcal{P}_{N}-\mathcal{Q}_{N}) v \rangle_{\mathbb{H}} =\\= \langle \xi_{c}, \overline{\mu}\mathcal{P}_{k, N} v \rangle_{\mathbb{H}} - \langle \xi_{c}, \overline{\mu}\mathcal{Q}_{k,N} v \rangle_{\mathbb{H}} + \langle \xi_{c}, \overline{\mu}(\mathcal{P}_{N}\mathcal{I}_{k,N}-\mathcal{Q}_{N}\mathcal{I}_{k,N}) v \rangle_{\mathbb{H}} \geq \\ \geq -|P_{k,N}\xi_{c}|_{\mathbb{H}} \cdot \overline{\mu} |\mathcal{P}_{k,N}v|_{\mathbb{H}} - |\mathcal{Q}_{k,N}\xi_{c}|_{\mathbb{H}} \cdot \overline{\mu} |\mathcal{Q}_{k,N}v|_{\mathbb{H}} - |\mathcal{I}_{k,N}\xi_{c}|_{\mathbb{H}}\cdot \overline{\mu} |\mathcal{I}_{k,N}v|_{\mathbb{H}} \geq \\ \geq
			-(\tau_{3}-\varepsilon) |\xi_{c}|^{2}_{\mathbb{H}}  -\frac{\overline{\mu}^{2}}{4(\tau_{3}-\varepsilon)}|v|^{2}_{\mathbb{H}}.
		\end{split}
	\end{equation}
	Analogously, we write
	\begin{equation}
		\begin{split}
			\langle \xi_{I}, \overline{\mu}(\mathcal{P}_{N}-\mathcal{Q}_{N}) v \rangle_{\mathbb{H}} = \\ =  \langle \xi_{I}, \overline{\mu}\mathcal{P}_{k,N}v \rangle_{\mathbb{H}} - \langle \xi_{I}, \overline{\mu}\mathcal{Q}_{k,N}v \rangle_{\mathbb{H}} + \langle \xi_{I}, \overline{\mu}(\mathcal{P}_{N}\mathcal{I}_{k,N}-\mathcal{Q}_{N}\mathcal{I}_{k,N}) v\rangle_{\mathbb{H}}
		\end{split}
	\end{equation}
	and use the estimates
	\begin{equation}
		\label{EQ: SpatAvrgNonOscXiIprojPart}
		\begin{split}
			\langle \xi_{I}, \overline{\mu}\mathcal{P}_{k,N}v \rangle_{\mathbb{H}} - \langle \xi_{I}, \overline{\mu}\mathcal{Q}_{k,N}v \rangle_{\mathbb{H}} \geq \\ \geq- (\tau_{2} - \varepsilon)|(\mathcal{P}_{k,N} + \mathcal{Q}_{k,N}) \xi_{I}|^{2}_{\mathbb{H}} - \frac{4\Lambda^{2}}{\theta_{\varepsilon}} |(\mathcal{P}_{k,N} + \mathcal{Q}_{k,N}) v|^{2}_{\mathbb{H}},
		\end{split}
	\end{equation}
	where $\tau_{2} \theta_{\varepsilon} = \tau_{2} - \varepsilon$, i.e. $\theta_{\varepsilon} = 1 - \varepsilon/\tau_{2}$,
	and
	\begin{equation}
		\label{EQ: SpatAvrgNonOscXiIinterPart}
		\begin{split}
			\langle \xi_{I}, \overline{\mu}(\mathcal{P}_{N}\mathcal{I}_{k,N}-\mathcal{Q}_{N}\mathcal{I}_{k,N}) v\rangle_{\mathbb{H}} =\\= \langle \mathcal{I}_{k,N}\xi_{I} - a (q)\mathcal{I}_{k,N}v, \overline{\mu}(\mathcal{P}_{N}\mathcal{I}_{k,N}-\mathcal{Q}_{N}\mathcal{I}_{k,N}) v\rangle_{\mathbb{H}} +\\+ \langle  a (q)\mathcal{I}_{k,N}v, \overline{\mu}(\mathcal{P}_{N}\mathcal{I}_{k,N}-\mathcal{Q}_{N}\mathcal{I}_{k,N}) v\rangle_{\mathbb{H}} \geq \\ \geq - (\tau_{1} - \varepsilon)|\mathcal{I}_{k,N}\xi_{I} - a(q) \mathcal{I}_{k,N}v|^{2}_{\mathbb{H}} - \frac{\overline{\mu}^{2}}{4(\tau_{1}-\varepsilon)}|\mathcal{I}_{k,N}v|^{2} +\\+ \overline{\mu}a(q) \left(|(\mathcal{P}_{N}\mathcal{I}_{k,N}v|^{2}_{\mathbb{H}}-|\mathcal{Q}_{N}\mathcal{I}_{k,N}v|^{2}_{\mathbb{H}}\right).
		\end{split}
	\end{equation}
	Moreover, it is not hard to see that
	\begin{equation}
		\label{EQ: SpatAvrgNonOscXiIremainderPart}
		\begin{split}
			\overline{\mu}a(q) \left(|(\mathcal{P}_{N}\mathcal{I}_{k,N}v|^{2}_{\mathbb{H}}-|\mathcal{Q}_{N}\mathcal{I}_{k,N}v|^{2}_{\mathbb{H}}\right) \geq \\ \geq \langle a(q) v, \overline{\mu}(\mathcal{P}_{N} - \mathcal{Q}_{N})v \rangle_{\mathbb{H}} - \overline{\mu}(\Lambda + \delta) |(\mathcal{P}_{k,N} + \mathcal{Q}_{k,N})v|^{2}_{\mathbb{H}}.
		\end{split}
	\end{equation}
	
	Combining \eqref{EQ: SpatAvrgNonOscLinearPart}, \eqref{EQ: SpatAvrgNonOscXiCPart}, \eqref{EQ: SpatAvrgNonOscXiIprojPart}, \eqref{EQ: SpatAvrgNonOscXiIinterPart} and \eqref{EQ: SpatAvrgNonOscXiIremainderPart}, we obtain
	\begin{equation}
		\begin{split}
			\langle Av - a(q)v + \xi_{I} + \xi_{c}, \overline{\mu}(\mathcal{P}_{N} - \mathcal{Q}_{N})v \rangle_{\mathbb{H}} + \mathcal{F}^{(\varepsilon)}_{q}(v,\xi) \geq \\ \geq \left(\overline{\mu}^{2} - \delta^{2}(\tau_{1}-\varepsilon) - \Lambda^{2}(\tau_{2}-\varepsilon) - \frac{\overline{\mu}^{2}}{4(\tau_{3}-\varepsilon)} - \frac{\overline{\mu}^{2}}{4(\tau_{1}-\varepsilon)} \right) |\mathcal{I}_{k,N} v|^{2}_{\mathbb{H}} + \\+
			\left(\overline{\mu}^{2} + \overline{\mu} k - \Lambda^{2}(\tau_{3}-\varepsilon) - \frac{\overline{\mu}^{2}}{4(\tau_{3}-\varepsilon)} - \frac{4\Lambda^{2}}{\theta_{\varepsilon}} - \overline{\mu} (\Lambda + \delta) \right) |(\mathcal{P}_{k,N} + \mathcal{Q}_{k,N})v|^{2}_{\mathbb{H}}.
		\end{split}
	\end{equation}
	Since the expressions in brackets depend continuously on $\varepsilon$, it is sufficient to establish their positivity for $\varepsilon=0$. This leads to the conditions
	\begin{equation}
		\frac{\overline{\mu}^{2}}{4} - \delta^{2} > 0 \text{ and } \overline{\mu}k + \frac{3}{4}\overline{\mu}^{2} - 5\Lambda^{2} - \overline{\mu}(\Lambda + \delta) > 0.
	\end{equation}
	Here the first inequality is as in \eqref{EQ: SpatAvrgNonOscillConditions}. In particular, we have $\frac{3}{4}\overline{\mu}^{2} - \overline{\mu}\delta > 0$ and the second inequality can be weakened to 
	\begin{equation}
		k - \frac{5 \Lambda^{2}}{\mu} - \Lambda \geq 0
	\end{equation}
	which is exactly the second inequality in \eqref{EQ: SpatAvrgNonOscillConditions}. The proof is finished.
\end{proof}

\begin{remark}
	\label{REM: ImprovedSpatAvrgCondsNonoOscillation}
	It is not hard to see that the inequalities from \eqref{EQ: SpatAvrgNonOscillConditions} are satisfied under \eqref{EQ: SpatAvrgInequalitiesZelik}.
\end{remark}

\begin{theorem}
	\label{TH: SpatAveragingNonOscill}
	Let the conditions of Theorems \ref{TH: SpatAvrgLagrangeBundle} and \ref{TH: SpatAvrgAuxiliaryCone} be satisfied. Then the nonoscillation condition \eqref{EQ: NonoscillationSpatAvrg} is satisfied and
	\begin{equation}
		\sup_{q \in \mathcal{Q}}\| P(q)\|_{\mathcal{L}(\mathbb{H})} \leq \delta^{-1}_{V},
	\end{equation}
	where $\delta_{V}$ as in \eqref{EQ: SpatAvrgSingFuncIntegralCond}. In other words, the Hamiltonian system \eqref{EQ: NonStatHamiltonianSystemSpatAv} is uniformly nonoscillating.
	
	Moreover, there exists $\delta_{V} > 0$ such that the family of quadratic forms $V_{q}(v) := \langle v, P(q)v \rangle_{\mathbb{H}}$, where $q \in \mathcal{Q}$, satisfies
	\begin{equation}
		V_{\vartheta^{T}(q)}(v(T)) - V_{q}(v_{0}) + \int_{0}^{T}\mathcal{F}_{\vartheta^{t}(q)}(v(t),\xi(t))dt \geq \delta_{V} \int_{0}^{T} (|v(t)|^{2}_{\mathbb{H}} + |\xi(t)|^{2}_{\mathbb{U}})dt
	\end{equation}
	for any $T>0$ and any solution $v(\cdot)$ satisfying \eqref{EQ: SpatAvrgControlSystemNonoscExpl} on $[0,T]$ with $\xi(\cdot) \in L_{2}(0,T;\mathbb{H})$. In particular, for any $q \in \mathcal{Q}$, $V_{q}(v) < 0$ for any nonzero $v \in \operatorname{Ran}\mathcal{Q}_{N}$ and $V_{q}(v) > 0$ for any nonzero $v \in \operatorname{Ran}\mathcal{P}_{N}$.
\end{theorem}
%------
% Insert acknowledgments and information
% regarding funding at the end of the last
% section, i.e., right before the bibliography.
%------

%\begin{ack}
%We thank X.
%\end{ack}

\begin{funding}
The reported study was funded by the Russian Science Foundation (Project 22-11-00172).
\end{funding}

\section*{Data availability}
Data sharing not applicable to this article as no datasets were generated or analyzed during the current study.

\section*{Conflict of interest}
The author has no conflicts of interest to declare that are relevant
to the content of this article.

%------
% Insert the bibliography.
%------

\end{document}